\documentclass[12pt]{amsart}

\usepackage[draft]{say}

\usepackage{graphicx}
\usepackage{amssymb}
\usepackage{amsmath, amscd}
\usepackage{dbnsymb}
\usepackage
{hyperref}

\usepackage{physics}
\usepackage{amsmath}
\usepackage{tikz, tikz-cd}
\usepackage{mathdots}
\usepackage{yhmath}
\usepackage{cancel}
\usepackage{color}
\usepackage{siunitx}
\usepackage{array}
\usepackage{multirow}
\usepackage{amssymb}
\usepackage{gensymb}
\usepackage{tabularx}
\usepackage{extarrows}
\usepackage{booktabs}
\usetikzlibrary{fadings}
\usetikzlibrary{patterns}
\usetikzlibrary{shadows.blur}
\usetikzlibrary{shapes}

\newtheorem{thm}{Theorem}[section]

\newtheorem{theorem}[thm]{Theorem}
\newtheorem{proposition}[thm]{Proposition}
\newtheorem{corollary}[thm]{Corollary}
\newtheorem{lemma}[thm]{Lemma}

\theoremstyle{definition}
\newtheorem{definition}[thm]{Definition}

\newtheorem{construction}[thm]{Construction}

\theoremstyle{remark}
\newtheorem{remark}[thm]{Remark}

\newtheorem{example}[thm]{Example}

\newcommand{\R}{{\mathbb{R}}}
\newcommand{\C}{{\mathbb{C}}}
\newcommand{\Q}{{\mathbb{Q}}}
\newcommand{\Z}{{\mathbb{Z}}}

\usepackage[margin=1in,marginparwidth=0.8in, marginparsep=0.1in]{geometry}

\usepackage{epigraph}

\usepackage{wrapfig, framed, caption}

\begin{document}

\title{Quantum mirrors of cubic planar graph Legendrians} 
\author{Matthias Scharitzer}
\address{Centre for Quantum Mathematics, SDU, Campusvej 55, 5230 Odense M, Denmark}
\email{matthias.scharitzer@gmail.com} 
\author{Vivek Shende}
\address{Centre for Quantum Mathematics, SDU, Campusvej 55, 5230 Odense M, Denmark 
$\qquad \qquad$ \& Department of Mathematics, UC Berkeley, Evans Hall, Berkeley CA 94720, USA}
\email{vivek.vijay.shende@gmail.com}

\maketitle

\begin{abstract}
For a certain class of Legendrian surfaces in the five-sphere, associated to cubic planar graphs, 
we show that the all-genus skein-valued holomorphic curve invariants of any filling
are annihilated by certain explicit skein-valued operator equations.  
\end{abstract}

\thispagestyle{empty}

\section{Introduction}

By {\em quantum} mirror symmetry, we mean the phenomenon where the all genus partition function
on the A-model side of mirror symmetry is a `wave function for', i.e. is annihilated by, 
certain operators which quantize the moduli space
of the B-model mirror.  In the closed topological string theory, such a relation is derived physically
from the holomorphic anomaly equation \cite{BCOV, Witten-background}.  For the open topological
string, such a relation
was derived physically \cite{Aganagic-Vafa, ADKMV, AV2, AENV} from consideration of the quantum 
Chern-Simons theory which is carried by Lagrangian branes \cite{Witten}.  
The power of
(quantum) mirror symmetry is that (the quantization of) the B-model is mathematically simpler than,
and yet often uniquely or nearly uniquely
determines, the (all genus) A-model partition function.  However, mathematical proofs  often
proceed in the reverse direction:  first computing or nearly computing the all genus partition function, then showing
it satisfies the desired recursion.  Some  mathematical
results of this nature include \cite{Eynard-Orantin-remodeling, FLZ-remodeling, Bousseau}. 

Given that the physical arguments for quantum mirror symmetry are built from Chern-Simons theory, 
it is notable in this context that it has recently been understood how to capture in mathematics some of the relation
between open topological string and Chern-Simons theory \cite{SOB}.  
The fundamental observation is that the boundaries of the moduli
space of holomorphic maps from curves-with-boundary can be grouped together in such a way that 
in each group, we meet a collection of curves-with-boundary whose boundaries themselves satisfy 
the HOMFLYPT skein
relation.  Hence: one can invariantly count holomorphic curves 
by the classes of their boundaries in the skein module of the Lagrangian.  

In fact, this skeins-on-branes formalism allows a transparent {\em a priori} derivation of quantum mirror symmetry, 
at least in the case when the Calabi-Yau 3-fold $X$ is noncompact with convex end, and the Lagrangian $L\subset X$
is asymptotic to a Legendrian submanifold of the ideal contact boundary which is Reeb-positive \cite{unknot}.  
The idea is the following.  The coefficients of the equations that cut out the mirror moduli
space (the ``augmentation variety'') are themselves counts of holomorphic curves
in a symplectization $(\partial X \times \R, \partial L \times \R)$.  Thus, the assertion that some quantization 
annihilates the curve counting invariants can be cashed out into an assertion that a certain (signed) count of holomorphic curves 
is zero.  The usual geometric way to establish such an identity is to exhibit said count as the boundary of some 
one-dimensional moduli space.  Here the relevant space parameterizes curves in $X$ with boundary along $L$ and asymptotic to an index one 
Reeb chord of $\partial L$.  

In \cite{unknot} this method was used to treat
perhaps the simplest example: the so-called Harvey-Lawson brane in $\C^3$,
which is a smoothing of the positive real cone through the torus
$$(e^{i \theta_1}, e^{i \theta_2}, e^{i \theta_3}) \qquad \qquad \theta_1 + \theta_2 + \theta_3 = 0$$
Counting holomorphic curves ending on said Lagrangian is a problem long
studied both in the string theory \cite{Aganagic-Vafa,  AKV} and math \cite{Liu, Katz-Liu, Fang-Liu} literature, and serves as a 
starting point for many generalizations \cite{AKMV, AENV}.  A distinctive feature of the work \cite{unknot} 
is that one {\em first} proves quantum mirror symmetry and {\em then} solves the operator equation for the partition function.

\vspace{2mm} 

In \cite{Treumann-Zaslow}, Treumann and Zaslow consider asymptotically conic 
Lagrangians in $\C^3$ defined
by taking boundary connect sums of Harvey-Lawson branes.   
The connect sums in
question are organized combinatorially by tetrahedronizations of a real 3-ball.  
Let us summarize the construction.
For an appropriate choice of identification $\C^3 = T^* \R^3$, the ideal boundary of
the Harvey-Lawson cone is a Legendrian surface in $S^5$ which 
projects 2:1 to the ideal boundary $S^2$ of $\R^3$, branched along four points.  
This ideal boundary can be described as a Legendrian in $J^1 S^2 \subset S^5$
via a front projection to $S^2 \times \R$.  The front projection is an embedding away from a tetrahedron graph 
in $S^2 \times 0$, and is 2:1 along the interiors of the edges.   The construction of \cite{Treumann-Zaslow}
is essentially to glue this  description along the faces of the {\em dual} tetrahedron.  The combinatorics
of the construction is captured by a trivalent graph on the $S^2$ filled by a `foam' with certain properties. 

\vspace{2mm} 

In the present article, we determine the quantum mirrors for these Legendrians.  As in \cite{unknot}, 
these are skein-valued operator equations which annihilate the skein-valued partition function
of any Lagrangian filling of the Legendrian boundary. 

\vspace{2mm}

We do not solve this operator equation.  However, the specialization of our operator to the $U(1)$ skein, 
i.e. an appropriate quantum torus, was previously discovered by  Schrader, Shen, and Zaslow, 
who conjectured a corresponding quantum mirror symmetry \cite{Zaslow-wave, Schrader-Shen-Zaslow}.
They moreover solved the operator equation (in the specialization, a q-difference equation) for a certain
class of fillings.  

Let us explain briefly how \cite{Schrader-Shen-Zaslow} arrived at, and solved, this equation.   Previously in \cite{Treumann-Zaslow},
it was observed that flips of trivalent graphs correspond to disk surgeries of Legendrians, and
the corresponding moduli of objects transform according to the cluster transformations for the cluster structure on
the moduli of rank two local systems on the sphere.\footnote{A similar phenomenon was previously 
observed in \cite{STWZ, STW}, though in that setting the corresponding moduli {\em were} 
the cluster charts $(\C^*)^{2g}$, whereas in 
\cite{Treumann-Zaslow} the moduli  were Lagrangians in said cluster charts.  The relation is explained in \cite{Casals-Li}.
The same cluster structure arising from a different view of the same physical system
appears also in \cite{CEHRV, DGGo}. 
}
In \cite{Schrader-Shen-Zaslow} it was shown that the whole picture quantizes 
compatibly with the quantum cluster algebra of \cite{Fock-Goncharov-quantum}, which, moreover,
provides solutions of the resulting q-difference equations.  Thus, on the physical grounds mentioned above, 
\cite{Schrader-Shen-Zaslow} conjecture that these q-difference equations should annihilate an open Gromov-Witten
partition function.  

\vspace{2mm}

In the present article, cluster structures will not appear.  However, if one wished to solve the operator
equation in the full skein, one plausible approach would be to upgrade the cluster algebra calculations carried out
in \cite{Schrader-Shen-Zaslow} to some (presently unknown) skein-valued cluster algebra.  
It would also be useful to have an explicit 
algebraic expression for the action of $Sk(\partial L)$ on $Sk(L)$, generalizing the genus one formulas of  \cite{Morton-Samuelson}.
Finding either seems a promising direction of future research. 

\vspace{2mm}
{\bf Recent developments.} Since the first appearance of this article, the question in the previous paragraph has been explored by curve counting methods in \cite{SS-2} and algebraically in \cite{HSZ}. 

\vspace{2mm}
{\bf Acknowledgements.} 
We thank Roger Casals, Tobias Ekholm, and Peter Samuelson for helpful conversations.  
The work presented in this article is supported by Novo Nordisk Foundation grant NNF20OC0066298, 
Villum Fonden Villum Investigator grant 37814, and Danish National Research Foundation grant DNRF157.  

\vspace{4mm}

\section{The face relation of \cite{Schrader-Shen-Zaslow}, formulated in the linking skein} \label{sec: face} 

In  \cite[Eq. 4.3.3]{Schrader-Shen-Zaslow} appears a certain element of a certain quantum torus. 
Here we recall the relation of quantum tori with linking skeins, and then describe their element skein
theoretically.  We use $q^{1/2}$ where \cite{Schrader-Shen-Zaslow} use $q$. 

Recall that to any lattice $N$ with an antisymmetric pairing, the associated quantum torus
is the quotient of the free non-commutative polynomial algebra over $\Z[q^{\pm 1/2}]$ 
with generators $[n]$ for $n \in N$ by 
the relations
$$q^{- (n,m)/2} [n][m] = [n+m] = q^{- (m,n)/2} [m][n]$$ 

Quantum tori arise naturally as linking skeins of surfaces; we recall
now this identification. 
To any oriented 3-manifold $M$, the {\em linking skein} $Lk(M)$ \cite{Przytycki} is the  module over $\Z[q^{\pm 1/2}]$ generated by formal
$\Z[q^{\pm 1/2}]$ linear combinations of oriented framed links, 
subject to the following skein relations:  
$$\overcrossing = q^{1/2} \,\, \smoothing \qquad \qquad \qquad \bigcirc = 1$$
A choice of framed links giving a basis for $H_1(M, \Z)$ determines an isomorphism
$$\Z[q^{\pm 1/2}][H_1(M, \Z)] \xrightarrow{\sim} Lk(M)$$   
In particular, under $Lk(S^3) \xrightarrow{\sim} \Z[q^{\pm 1/2}]$, 
a link $L$ is sent to a power of $q$ given by the `total linking number' of $L$, hence the name. 

For a surface $\Sigma$, the linking skein $Lk(\Sigma) := Lk(\Sigma \times \R)$ acquires an algebra
structure with product given by concatenating
$\R \sqcup \R \to \R$.  In fact,  $Lk(\Sigma)$ is the quantum torus associated to 
the lattice $H_1(\Sigma, \Z)$  with skew symmetric form coming from the Poincar\'e pairing. 
Representative elements can be constructed as follows.  For a simple closed oriented multi-curve 
$C \subset \Sigma$, we write $[C] \in Lk(\Sigma)$ to mean $C \times 0 \subset \Sigma \times 0 \subset \Sigma \times \R$, 
framed by the positive vector in the $\R$ direction.  

If $M$ is a 3-manifold with boundary $\Sigma$, then there is similarly an action 
$$Lk(\partial M) \times Lk(M) \to Lk(M)$$ 
In case $M$ is a handlebody filling $\Sigma$, then 
the resulting action is the usual action of the quantum torus on a polynomial representation.

\begin{remark}
To 
see the quantum torus along with its polynomial representation 
`in coordinates', choose standard representatives of a symplectic 
basis of $H_1(\Sigma, \Z)$ compatible with the 
handlebody; i.e. some oriented simple closed curves on $\Sigma$, say
$A_1, \ldots, A_g$ and $B_1, \ldots, B_g$ on $\Sigma$, such that 
the $A_i$ are disjoint from each other, the $B_j$ are disjoint from each other, 
and $A_i \cdot B_j = 1$ given by a transverse single geometric intersection; 
such that moreover the $B_j$ become contractible in the handlebody $M$.   
Now $Lk(\Sigma)$ is identified with the ring freely generated by variables 
$A_i, B_j$ with $A_i B_i = q B_i A_i$ and all other commutators trivial; 
$Lk(M)$ is identified with $\Z[q^{\pm 1/2}][A_i]$; and the action is induced from
$B_i \cdot 1 = 1$.  
\end{remark}


We now discuss the quantum tori which appear in \cite{Schrader-Shen-Zaslow} and
their skein theoretic interpretations.  Let us first recall their general setup. 
Fix once and for all an orientation on $S^2$.  
Fix a trivalent graph $\Gamma$ on $S^2$.  There is a Legendrian
$\Lambda_\Gamma \subset J^1 S^2$ such that $\pi: \Lambda_\Gamma \to S^2$  is  2:1 and ramified only 
at the vertices of $\Gamma$, and such that the front projection $\Gamma \to S^2 \times \R$ 
is an embedding away from the edges of $\Gamma$.  

For a face $F$ of $\Gamma$, we write $F^+ \subset \Lambda_\Gamma$ for the `upper' sheet of
$\pi^{-1}(F)$ and $F^-$ for the lower sheet; here upper and lower refers to the $\R$ coordinate of
the front projection.  
Note $\pi^{-1}(\Gamma)$ determines a graph on the oriented surface $\Lambda_\Gamma$, the faces of which
are bicolored by $+$ and $-$.  We fix a point $p_F$ in each face $F$, and write
$p_F^+$ and $p_F^-$ for its preimages in $F^+$ and $F^-$. 

We orient the edges 
of $\pi^{-1}(\Gamma)$ compatibly with the boundary orientation of $+$ faces, and anticompatibly
with the boundary orientation of $-$ faces.
If $E$ is an edge of $\Gamma$, then $\widetilde{E} := \pi^{-1}(E)$ is a simple closed curve on $\Lambda_\Gamma$, 
comprised of two edges of $\pi^{-1}(\Gamma)$, which may be oriented by the convention above.  These
orientations are compatible and determine an orientation of $\widetilde{E}$. 

In \cite{Schrader-Shen-Zaslow} appear two quantum tori associated to $\Gamma$. 
The first is denoted $\mathcal{T}^q_{\Gamma}$, and is
the quantum torus associated to the lattice formally generated by the edges,
$E$ of the graph $\Gamma$, where the pairing is inherited from the intersection
pairing of the $\widetilde{E}$.  From this definition and the properties of the linking skein, it is apparent
that there is an embedding:
\begin{eqnarray*}
\mathcal{T}^q_{\Gamma} & \hookrightarrow & Lk(\Lambda_\Gamma   \setminus \bigcup_F p_F^{\pm}) \\
\, [E] & \mapsto & [\widetilde{E}]
\end{eqnarray*}

The second quantum torus is the quotient of $\mathcal{T}^q_{\Gamma}$ by certain relations, which
we will express in the skein. 
For a face $F$ of $\Gamma$, we write a curve traveling the positively oriented boundary of $F^+$ 
as $\partial F^+$, and a curve traveling the negatively oriented boundary of $F^-$ as $\partial^- F^-$.  
Then: 
\begin{equation} \label{twisted quotient}
\mathcal{T}^q_{\underline{\Gamma}} := 
\frac{Lk(\Lambda_\Gamma   \setminus \bigcup_F p_F^{\pm})}{(q^{1/2} + [\partial F^+], q^{1/2} +  [\partial^- F^-])}
\end{equation} 

In order to track the relation introduced in \eqref{twisted quotient} skein theoretically, we introduce
new objects into the skein, called framing lines and sign lines. 
A framing line $\mathfrak{l} \subset M$ is a piecewise smooth integer $1$-chain (e.g. oriented, unframed link with multiplicity).  
A sign line $\mathfrak{s} \subset M$ is a piecewise smooth $\Z/2\Z$ $1$-chain (e.g. unoriented, unframed link). 

\begin{definition}
We write 
$Lk(M, \mathfrak{l}, (-1)^{\mathfrak{s}})$ 
for the linking skein of links in $M \setminus \mathfrak{l}$, modulo the 
additional relation that crossing $\mathfrak{l}$ changes the framing of a link by the change of linking with $\mathfrak{l}$, 
and crossing $\mathfrak{s}$ multiplies by $-1$.  
If there are no framing lines, or no sign lines, then we omit the corresponding term from the notation. 

In case $M$ has boundary, we allow $\mathfrak{l}, \mathfrak{s}$ to meet the boundary transversely. 
For surfaces with appropriate $0$-chains $\mathfrak{p}, \mathfrak{q}$, we write 
$Lk(\Sigma, \mathfrak{p}, (-1)^{\mathfrak{q}}) := Lk(\Sigma \times \R, \mathfrak{p} \times \R, 
(-1)^{\mathfrak{q} \times \R})$.  
Again $Lk(\partial M, \partial \mathfrak{l}, (-1)^{\partial \mathfrak{s}})$ acts on 
$Lk(M, \mathfrak{l}, (-1)^{ \mathfrak{s}})$ by concatenation.  
\end{definition}

Thus if we set
$$\mathfrak{p}_{\Gamma} = \sum_F p_{F}^+ - p_{F}^-$$
then
\begin{equation}
\mathcal{T}^q_{\underline{\Gamma}} = 
Lk(\Lambda_\Gamma, \mathfrak{p}_{\Gamma},  (-1)^{\mathfrak{p}_{\Gamma}}).
\end{equation}

The skeins $Lk(\Lambda_\Gamma)$ and 
$\mathcal{T}^q_{\underline{\Gamma}} = Lk(\Lambda_\Gamma, \mathfrak{p}_{\Gamma},  (-1)^{\mathfrak{p}_{\Gamma}})$ are
(non canonically) isomorphic.   Indeed, 
let $\mathfrak{l}$ be any oriented tangle in $\Lambda_\Gamma \times \R$ such that
$\partial \mathfrak{l} = (\mathfrak{p}_{\Gamma} \times \infty) \sqcup (\emptyset \times -\infty)$.  
Such links exist, e.g. fix any paths on $\Lambda_\Gamma \times \infty$ connecting each 
$p_F^+$ to the corresponding $p_F^-$, 
and push them into $\Lambda_\Gamma \times \R$. 
Then $Lk(\Lambda_\Gamma \times \R, \mathfrak{l}, (-1)^{\mathfrak{l}})$ determines an isomorphism
given by ``pushing the skeins from the top to the bottom, through the sign and framing lines''. 
If $\mathfrak{l}$ arose from some paths $\ell$ as suggested above, 
then the isomorphism is computed by the formula
\begin{align*}
Lk(\Lambda_\Gamma, \mathfrak{p}_{\Gamma}, (-1)^{\mathfrak{p}_{\Gamma}}) & \to Lk(\Lambda_\Gamma) \\
[C] & \mapsto (-q^{1/2})^{\langle \ell, C\rangle} [C]
\end{align*}
where the pairing is the intersection pairing in $\Lambda_\Gamma \setminus \bigcup_F p_F^{\pm}$.  

\begin{remark}
That is, the above setup provides an isomorphism determined by a linear functional on
$H_1(\Lambda_\Gamma \setminus \bigcup_F p_F^{\pm}, \Z)$. 
Such an isomorphism, together with a choice of $A$ and $B$ cycles on $\Lambda_\Gamma$, 
is what is called a `framing' in  \cite{Schrader-Shen-Zaslow}. 
\end{remark}

If $L_\Gamma$ is a handlebody filling of $\Lambda_\Gamma$, 
and $\mathfrak{l}_\Gamma \subset L_\Gamma$ is a tangle with $\partial \mathfrak{l}_\Gamma = \mathfrak{p}_\Gamma$,
then  
$Lk(L_\Gamma, \mathfrak{l}_\Gamma, (-1)^{\mathfrak{l}_\Gamma})$ 
carries an action of 
$Lk(\Lambda_\gamma, \mathfrak{p}_\Gamma, (-1)^{\mathfrak{p}_{\Gamma}})$.  
Under the above isomorphism, this is identified with the polynomial representation of 
the quantum torus 

\vspace{2mm} 
Now let us (re)formulate the main conjecture of \cite{Schrader-Shen-Zaslow}. 

\begin{definition} \label{def: ell} 
Let $F$ be a face of $\Gamma$, fix a vertex $v = v_1$, and let 
$v_1, \ldots, v_n$ be the vertices of $F$ read in counterclockwise order.  
For $k = 2, \ldots, n$, let $\ell_k = \ell_{k}(F, v) $ be the loop which begins at $v_1$, travels in $F^+$  counterclockwise near $\partial F^+$ until
$v_k$, where it crosses from $F^+$ to $F^-$ at $v_k$, then travels clockwise on $F^-$ near $\partial F^-$ back 
to $v_1$.    
\end{definition} 

Note that if $E_1, \ldots, E_n$ are the edges of $F$, indexed so that $E_i$ meets $v_i$ and $v_{i+1}$, then
$$\ell_k = \widetilde{E}_1 + \cdots + \widetilde{E}_{k-1} \in H_1(\Lambda_\Gamma)$$

\begin{definition} \label{def: R}
(\cite[Eq. 4.3.3]{Schrader-Shen-Zaslow} rewritten in skein) 
The face relation associated to $F$ is 
\begin{eqnarray*} 
R^q_{\Gamma, F, v} & = & q^{-1/2} \,\,\,\,\,\,\,\,\, + [\ell_2] + [\ell_3] + \cdots + [\ell_n] \in Lk(\Lambda_\Gamma, \mathfrak{p}_{\Gamma}, (-1)^{\mathfrak{p}_{\Gamma}}) \\
& = & q^{-1/2} [\bigcirc] + [\ell_2] + [\ell_3] + \cdots + [\ell_n] \in Lk(\Lambda_\Gamma, \mathfrak{p}_{\Gamma}, (-1)^{\mathfrak{p}_{\Gamma}})
\end{eqnarray*} 
\end{definition}

As noted in \cite{Schrader-Shen-Zaslow}, one can check that the ideal generated by $R_{\Gamma, F, v}$ does not depend on the choice of $v \in F$. 

The main conjecture of Schrader, Shen, and Zaslow in \cite[Sec 6.3]{Schrader-Shen-Zaslow} can then be translated to the statement 
that there should be some (by those authors unspecified) theory of open Gromov-Witten invariants, 
with the following properties: 
\begin{itemize}
\item For certain $L_\Gamma \subset \C^3$ constructed in that article with $\partial L_\Gamma = \Lambda_\Gamma$, 
there is an invariant
$$\Psi_{L_\Gamma} = 1 + \cdots \in Lk(L_\Gamma, \mathfrak{l}_\Gamma, (-1)^{\mathfrak{l}_\Gamma})$$
\item The invariant is annihilated by all face operators: 
$$ R^q_{\Gamma, F, v} \Psi_{L_\Gamma} = 0$$
\end{itemize}

\section{Skein valued curve counting and relations from contact infinity}  \label{sob story}

In this section we recall first the skein-valued open curve counting of  \cite{SOB} and, secondly, how, at least
in special cases, one can show that the count of curves ending on an asymptotically cylindrical Lagrangian $L$
is annihilated by a quantization of 
the augmentation variety of $\partial L$ \cite{unknot}. 

Recall that the HOMFLYPT framed skein relations are:

\vspace{2mm}
\begin{center}
\includegraphics[scale=0.2]{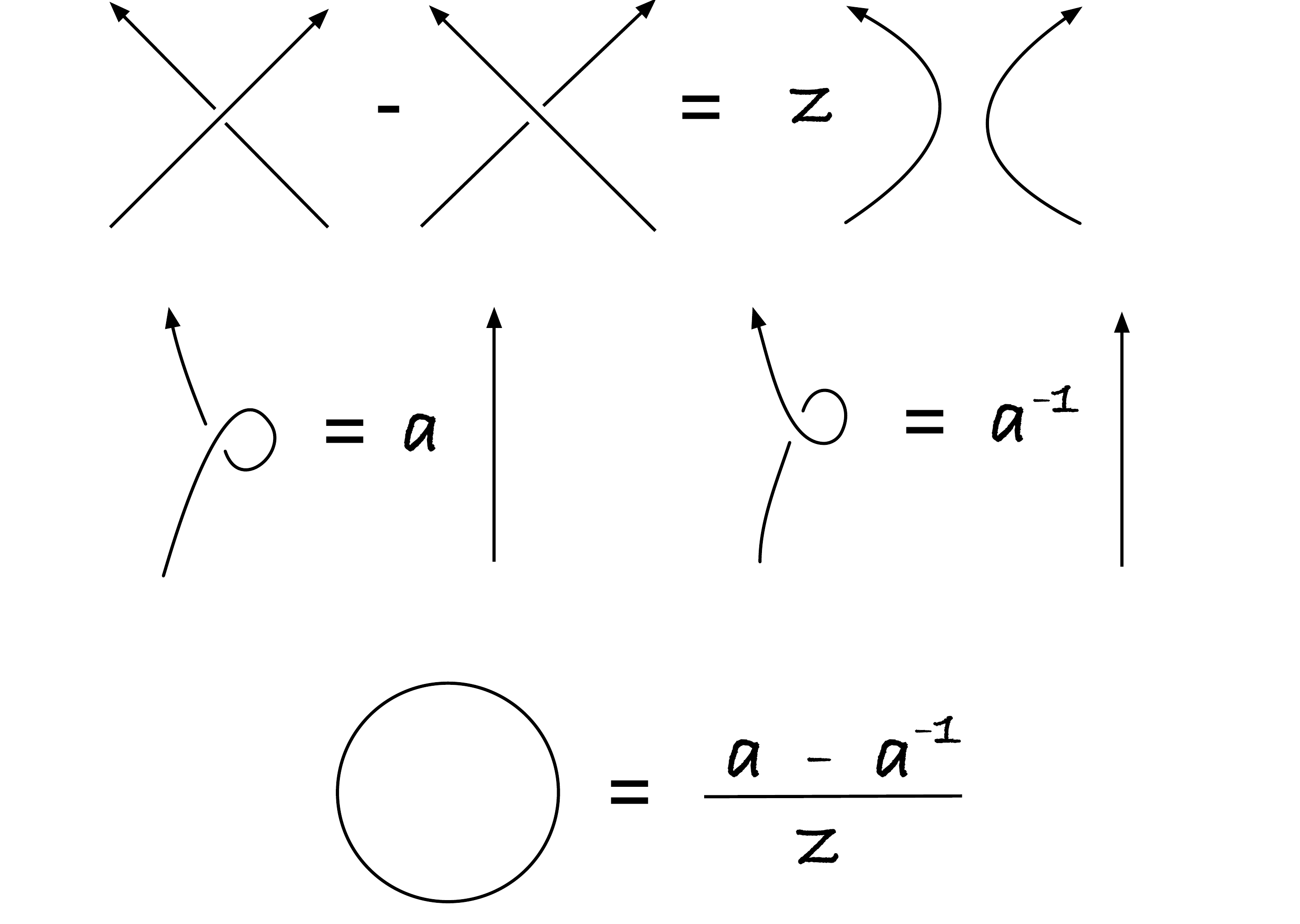}
\end{center} 
\vspace{2mm}

For any oriented 3-manifold $M$, 
we write $Sk(M)$ for the quotient by the HOMFLYPT skein relations of the free $\Z[a^{\pm}, z^{\pm}]$ module
generated by isotopy classes of oriented framed links.  
Note the linking skein relations imply the HOMFLYPT skein relations specialized at 
$a = q^{1/2}$ and $z = q^{1/2} - q^{-1/2}$.  Thus $Lk(M)$ is a quotient of $Sk(M)$.\footnote{This
corresponds to the fact that 
the HOMFLYPT skein organizes the Wilson lines for $U(n)$ Chern-Simons theory for unspecified $n$; the linking skein
corresponds to $n=1$.  See \cite{Gelca-theta} for a book-length discussion of the $U(1)$ theory.} 
Just as for the linking skein, one may introduce framing lines $\mathfrak{l} \subset M$
and consider $Sk(M, \mathfrak{l})$.  It is also possible to introduce sign lines,  though we postpone
discussion of their role in curve-counting until Section \ref{sec: spin sign}.

Let $X$ be a Calabi-Yau complex 3-fold, $L \subset X$ a Lagrangian with vanishing Maslov class. 
The article \cite{SOB} defines a count of holomorphic curves in $X$ with boundary on $L$ 
taking values in the framed HOMFLYPT skein module $\widehat{Sk}(L) := \Q \otimes Sk(L)[[H_2(X, L), z^{-1}]]$, where
the completion is along the cone of classes positive with respect to the symplectic form.\footnote{In case $H_2(X)$ 
has no classes on which the symplectic form is positive, e.g. $X$ exact, 
the class in $H_2(X, L) \hookrightarrow H_1(L)$ is already determined by the skein element,
so we may instead complete the skein.  Likewise a choice of splitting of $H_2(X) \to H_2(X, L)$ allows
to reduce coefficients to  $Sk(L)[[H_2(X)]]$.} 
(More precisely, in \cite{SOB} this is shown conditional on the existence of a perturbation setup
satisfying certain axiomatics; the required setup is constructed in the foundational works \cite{ghost, bare}.) 

The idea that holomorphic curves should be counted by their boundaries in the skein is in some sense
implicit in the belief  that Lagrangian branes in topological string theory
carry Chern-Simons fields, and that strings ending on these branes introduce line operators.  The string
arguments which establish this \cite{Witten} are carried out at the level of the path integrand and so are 
at present quite far from being made mathematically rigorous.  The argument in \cite{SOB} proceeds
instead directly at the level of moduli spaces of holomorphic curves.  The basic point 
is the following: while boundary bubbling phenomena 
would naively spoil the invariance of a numerical count of curves with boundary, in fact the boundary bubbling 
is always accompanied by the boundaries
of the holomorphic curves themselves transforming according to the skein relation.  So, it is well defined
to count curves weighted by the class of their boundaries in the skein module. 

Let us recall some unusual features of the curve counting setup of \cite{SOB, ghost, bare}.  First of all, one counts
possibly disconnected curves.  Second, the count is defined in terms of some auxilliary data: 

\begin{definition}
For $X$ a CY3-fold and $L \subset X$ a Lagrangian of vanishing Maslov class, 
a {\em brane structure} on $L$ is a spin structure, plus a 
vector field $v$ on $L$ with transverse zeros,
and a 4-chain $V$ such that $\partial V = 2L$ and $V = \pm J \cdot v$ near $\partial V$.  

We allow $V$
to meet $L$ also in its interior $V^\circ$ so long as it does so sufficiently transversely, in which case 
we will use $L \cap V^\circ$ as a framing line and consider 
$Sk(L, L\cap V^\circ)$.\footnote{In the current public version of \cite{SOB}, instead one is instructed
to trivialize $L \cap V^\circ$ by choosing a bounding 2-chain for it, but the arguments establish
the validity of the present mild generalization.}   

In case $X, L$ are noncompact with convex end, 
we demand $V, v$ to be asymptotically $\R$-invariant at the ideal boundary. 
\end{definition} 

The curve counts are invariant under deformation of brane structure, and 
depends on global change of $v, V$ only up to explicit monomial change of variable.  

Third, and most 
unusually compared to the usual Gromov-Witten
theory, while we work over a configuration space of maps from nodal curves, 
we perturb the holomorphic curve equation {\em only on irreducible components of curves 
whose images have nonzero symplectic area}.  Such components will then be embedded after appropriate generic
perturbation, but may still carry ghost bubbles and remain solutions.  In fact, however, 
a refinement of the Gromov compactness theorem shows that in 
0- and 1- dimensional moduli, the locus of curves {\em without} ghost bubbles is already compact \cite{SOB, ghost}.   
We consider only these curves, and count them by
$$\Psi_L := \sum_{C} \epsilon_C \cdot z^{-\chi(C)} \cdot a^{C \cap V} \cdot [\partial C] \in \widehat{Sk}(L)$$
Here, $\epsilon_C $ is a rational number (fractions arising from the use of multivalued perturbations), 
$\chi(C)$ is the Euler characteristic, $\partial C$ is framed using the vector field $v$, 
and $\widehat{Sk}(L)$ refers to the appropriate completion of $Sk(L, L \cap V^\circ)$ or more generally of $Sk(L, L \cap V^\circ)[H_2(X, L)]$. 

\begin{remark}
The arguments of the present article depend little on the details of the curve counting scheme in 
\cite{SOB}; one would expect them to apply to any skein-valued curve counting setup compatible with SFT stretching.  
The essential point of contact with \cite{SOB, ghost, bare} is  in the assertion that some such setup exists. 
\end{remark} 

We will set by definition $z = q^{1/2} - q^{-1/2}$.  

\begin{remark}
Because we do not count solutions with ghost bubbles, 
it is natural to expect from \cite{Pandharipande-degenerate} that $q$ compares to usual Gromov-Witten
variables as $q = e^{h}$, where $h$ is the Gromov-Witten genus counting parameter.  
For this reason we avoid referring to the skein-valued curve counting of \cite{SOB} 
as a ``Gromov-Witten'' theory.   (In particular, in terms of which variables appear, it is closer to Donaldson-Thomas theory.) 
\end{remark}

\vspace{2mm}

When $X, L$ are noncompact and eventually conical,
with convex end modeled on a contact manifold $\partial  X$ containing Legendrian $\partial L$, 
then the geometry at infinity provides a powerful tool to compute the skein-valued curve counts. 
For a contact manifold $Q$ and Legendrian $\Lambda$ of vanishing Maslov class,  
we say $(Q, \Lambda)$ is {\em Reeb positive} (resp. {\em non-negative}) if all Reeb orbits or chords have index $> 0$
(resp. $\ge 0$).  
Examples of non-negative $(\partial X, \partial L)$ are plentiful: for instance knot conormals all provide such Legendrians \cite{EENS}.
Positive examples are rarer: of knot conormals, only the unknot conormal is positive.  Nevertheless, it is useful to single out positivity 
because of the following consequence: 

\begin{lemma} \label{recursion}   \cite{unknot} 
Let $X$ be a symplectic manifold with convex end, and $L \subset X$ a Lagrangian asymptotic to a Legendrian
$\partial L \subset \partial X$.
Assume $(\partial X, \partial L)$ is Reeb positive.  Fix a brane structure for $L$, including vector field $v$ and 4-chain $V$.

For any Reeb chord $\rho$ of index $1$, let $$A_\rho := A_\rho(\partial X, \partial L) \in \widehat{Sk}(\R \times \partial L, \partial (L \cap V^{\circ}))$$
be the skein valued count of curves in the symplectization $(\R \times \partial X, \R \times \partial L)$, with boundary
on $\partial L$ and one positive puncture at $\rho$.\footnote{We choose a capping path for $\rho$, or alternatively, 
take $A_\rho$ valued in a skein with one entering and one exiting strand at $\infty \times \partial L$.}
Let $$\Psi = \Psi(X, L) \in \widehat{Sk}(L, L\cap V^{\circ})$$ be the count of holomorphic curves in $X$ ending on $L$.  
Then \begin{equation} \label{operator equation} A_\rho \Psi = 0 \end{equation} 
Here, $\widehat{Sk}(\R \times \partial L, \partial (L \cap V^{\circ}))$ acts on $\widehat{Sk}(L, L \cap V^\circ)$ by concatenation. 
\end{lemma} 
\begin{proof}
We recall the proof from \cite{unknot}: count curves at
the boundary of the 1-dimensional moduli space
of curves in $X$ with boundary on $L$ and one positive puncture at $\rho$.  
As a boundary, the count of such curves is zero.  By SFT compactness, 
\cite{BEHWZ}, the boundaries of this 1-dimensional moduli space come in two kinds: those 
coming from boundary breaking in the interior of $L$, and those associated to 
SFT breaking at infinity.  The term coming from boundary breaking is cancelled
by the skein-valued counting prescription.  By Reeb positivity, there 
are no possible SFT breakings along Reeb chords which will have expected dimension zero, so all
breakings at infinity must separate into some curve in the symplectization 
$(\partial X \times \R, \partial L \times \R)$ ending on the index one Reeb chord, 
disjoint union some curve in the interior of $(X, L)$.   The skein-valued count of such
disjoint unions is precisely $A_\rho \Psi$.  

Strictly speaking, the above argument requires SFT transversality in order to argue that
potential breakings of negative expected dimension do not in fact appear.  However, 
if $\rho$ is in fact a Reeb chord of minimal action, then such breakings are ruled out for action reasons, 
so SFT transversality is not required.  In this article we will always be able to ensure this lowest action condition. 
\end{proof}

Similar considerations in the non-negative case lead to similar but more complicated relations; 
see e.g. \cite{colored} for perhaps the simplest example.  
Beyond the non-negative case, virtual considerations would become essential.

For the unknot conormal, the $A_\rho$ of \eqref{operator equation} can be determined explicitly and solved 
for $\Psi$; the result agrees with
all-color HOMFLYPT invariant of the unknot \cite{unknot}.  
(Unfortunately, since other knot conormals are not positive, the result of \cite{unknot} exhausts the direct application of Lemma \ref{recursion}
to knot conormals.)  Also in \cite{unknot}, both $A_\rho$ and $\Psi$ are determined for the original Harvey-Lawson brane.  

We record the following simple criterion for positivity:

\begin{lemma} \label{satellite positivity}
Let $Q$ be a contact manifold, and suppose $\Lambda \subset Q$ is a Reeb-positive Legendrian.  Suppose
that $\Lambda' \subset J^1 \Lambda$ is Reeb-positive for the standard contact structure on the jet bundle.  
Then the Legendrian satellite $\Lambda' \# \Lambda \subset Q$ can be chosen Reeb-positive. 
\end{lemma} 
\begin{proof}
By rescaling the jet bundle, one sees the Reeb chords of the satellite converge either
to converge to chords of $\Lambda$, or 
points of $\Lambda$, in which case they were chords from  $\Lambda' \subset J^1 \Lambda$. 
\end{proof} 

\begin{remark}
In the situation of Lemma \ref{recursion}, the disks counted to give monomials of $A_\rho$ are precisely the same disks which 
give the monomials in the equations defining the augmentation variety of the Legendrian DGA of $\partial L$, as a subvariety
of the coefficient space $\mathrm{Spec}\, \C[H_2(X, L)]$.  
So the $A_\rho$ are a `skein valued' quantization of the augmentation
variety.  (If non-disk terms appear in $A_\rho$, they appear with higher power of $z$.)  
If one wants a quantization in the ordinary sense, i.e. some ideal in the quantum torus associated to $H_1(\partial L)$, then just take the specialization
of the HOMFLYPT skein $Sk(\partial L)$ to the linking skein $Lk(\partial L)$.
\end{remark}

\begin{remark}
Let us record an argument why the augmentation variety is Lagrangian.  
If $X$ is exact and subcritical (e.g. $X = \R^6$ as in the case of interest in this article), it follows from 
\cite{GPS2} that the wrapped Fukaya category of $(X, \partial L)$ is equivalent to the category of modules over endomorphisms
of the linking disk $\delta$ to $\partial L$.
As noted in \cite{GPS2}, there's an
inclusion of sectors $T^*(\partial L \times \R) \to (X, \partial L)$ carrying the cotangent fiber to the linking disk, hence by \cite{GPS1}
providing a map $C_* \Omega L \to End_X(\delta)$, along which we may pull back augmentations.  
Per \cite{Ekholm-Lekili} this structure is quasi-isomorphic
to the inclusion of the coefficients $C_* \Omega L$ in the Legendrian DGA of $\partial L$ with coefficients in $C_* \Omega L$.
(In the case at hand, $L$ is a surface of positive genus so $C_* \Omega L = \Z[\pi_1(L)]$, augmentations of
which factor through $\Z[H_1(L)]$.)
Returning to the Fukaya categorical formulation, now use \cite{GPS3} to pass to microsheaves, and
conclude from \cite{Shende-Takeda} that the map $C_* \Omega L \to End_X(\delta)$ is relative Calabi-Yau hence 
by \cite{Brav-Dyckerhoff} that we get a Lagrangian morphism on moduli spaces of augmentations, in the sense of \cite{PTVV}. 
Along the locus where everything is a smooth scheme of the expected dimension, such a Lagrangian morphism is just
a Lagrangian in the usual sense.  
\end{remark}

\section{Morse flow trees}

For $f: M \to \R$ a function, Floer strips between the graph of 
$tdf$ and the zero section will in some sense converge to Morse trajectories as $t \to 0$ \cite{Floer-morse};
corresponding disks governing $A_\infty$ structures correspondingly limit to Morse flow trees \cite{Fukaya-Oh}.  
More generally still and most relevant to us
is Ekholm's work \cite{E} on the analogous questions for Legendrians $\Lambda \subset J^1 M$.   Here we briefly review some results
of this work.

\vspace{2mm}
Recall a fatgraph is a graph, possibly with loops and half-edges, equipped with a cyclic order of edges at each vertex. 
To such a graph is canonically associated an oriented surface with boundary and boundary punctures, obtained by thickening
the vertices to disks, the edges to strips, and attaching the strips to the boundary of the disks using the cyclic order. 
Half-edges introduce punctures. 
For a fatgraph $\mathfrak{f}$, we write $S_\mathfrak{f}$ for this surface.  We write $\mathfrak{f}^\circ := \mathfrak{f} 
\setminus \{\mbox{vertices and half-edges}\}$. 
 
\begin{definition}\cite[Def. 2.10. (b) and (c)]{E} 
Let $M$ be a manifold and $\mathfrak{f}$ a fatgraph.  A {\em test graph}  is 
a commutative diagram of continuous maps 
$$
\begin{tikzcd}
\partial S_\mathfrak{f} \ar{r}{u} \ar{d}{\gamma} & J^1 M \ar{d}{\pi} \\ 
\mathfrak{f} \ar{r}{\eta}  & M
\end{tikzcd}
$$
such that $\gamma: \gamma^{-1}(\mathfrak{f}^\circ) \to \mathfrak{f}^\circ$ is a 2:1 cover, and, for any vertex $v \in \mathfrak{f}$, one has $|\gamma^{-1}(v)| = \mathrm{degree}(v)$.
\end{definition}

Let $ev: J^1 M \to \R$ be the projection recording the `value of the function'. 
Note that at a half-edge of the fatgraph, the corresponding $S_\mathfrak{f}$ has a boundary puncture.  
We say the puncture is positive if the orientation of the ideal boundary is carried by $ev$ to the orientation
of $\R$, and negative if it is carried to the opposite orientation.

\begin{definition} \cite[Def. 2.10. (a)]{E} 
Let $\Lambda \subset J^1 M$ be a Legendrian, and let $\Lambda^{reg} \subset \Lambda$ be the locus
where the base projection $\pi: \Lambda \to M$ is \'etale.  
Suppose given a test graph satisfying: 
$$
\begin{tikzcd}
\gamma^{-1}(\mathfrak{f}^\circ) \ar{r}{u} \ar{d} & \Lambda^{reg} \ar{d}  \\ 
\partial S_\mathfrak{f} \ar{r}{u} \ar{d}{\gamma} & \Lambda \ar{d}{\pi} \ar{r}{ev}&   \R \\ 
\mathfrak{f} \ar{r}{\eta}  & M
\end{tikzcd}
$$
Consider $p \in \gamma^{-1}(\mathfrak{f}^\circ)$.  Then the natural involution 
$\iota:  \gamma^{-1}(\mathfrak{f}^\circ) \to  \gamma^{-1}(\mathfrak{f}^\circ)$ extends to some
$Nbd(u(p)) \subset \Lambda^{reg}$. On possibly some smaller neighborhood, the base projection
$\pi$ is an isomorphism.  Thus on some $Nbd(\pi(u(p))$ we may consider the function 
\begin{equation}
\label{function difference} ev \circ \pi^{-1} - ev \circ \iota \circ \pi^{-1} 
\end{equation}
Fix now a metric on $M$.  
We say the test graph is a {\em flow graph} if $\pi \circ u|_{\gamma^{-1}(\mathfrak{f}^\circ)}$ 
is, at every $p$, a gradient flow line for the corresponding function (note $\partial S_\mathfrak{f}$ is oriented).  
\end{definition} 

We write $\Pi: J^1 M \to T^*M$ for the Lagrangian projection. 

\begin{theorem} \cite{E} \label{morse flow tree theorem}
Assume $\dim M = 2$ and that the front projection $(\pi \times ev) (\Lambda) \subset M \times \R$ has only crossings, 
cusp edges, and swallowtails, and these singularities are in mutually general position.  
Fix a metric $g$ on $M$.  

Then there is an isotopy of Legendrians $\Lambda_t$, where $\Lambda_0 = \Lambda$ and $\Lambda_t$ 
is a small perturbation of $t \Lambda$, such that: 
\begin{enumerate}
\item Any sequence ($t \to 0$) of holomorphic disks $D_t$ with boundary on $\Pi(\Lambda_t)$, possibly with punctures, 
has a subsequence converging to a Morse flow tree for $\Lambda$.  
\item Assume all rigid flow trees for $\Lambda$ with one positive punctures are transversely cut out.  Then
for all sufficiently small $t$, there is a bijection between rigid Morse flow trees with one positive puncture, and 
rigid holomorphic disks $D_t$ with boundary on $\Pi(\Lambda_t)$.
\end{enumerate} 
Moreover, metrics $g$ ensuring transversality of all rigid flow trees with one positive puncture are open and dense
in the space of metrics. 
\end{theorem} 

In addition to this result, Ekholm characterized rigidity of flow graphs in terms of an explicit finite list of allowable vertices.
For nice pictures, 
see \cite[Thm. 2.7]{EENS}.\footnote{Note that to match the fatgraph formulation above, we should view the bivalent
vertices of \cite[Thm. 2.7 (a), (b)]{EENS} rather as trivalent vertices with a half-edge attached to record the 
puncture, which is projected to the vertex under the map $\eta$.}

We will require, beyond the results explicitly stated in \cite{E}, the following two additional facts.   
The first concerns the $C^1$ limiting behavior of the boundary of the disks $D_t$ in (2).  
The proofs in \cite{E} in fact show $C^1$ convergence away from the vertices, 
and convergence $C^1$ close to some explicit model near the vertices.  This implies
in particular that if one has a vector field on $\Lambda$ which pairs positively with
the conormal cone to $u(\partial S_F)$, then the same will be true for the nearby $u(\partial D_t)$. 

The second concerns curves of higher genus.  The arguments of \cite{E} are local in nature, so 
also imply the analogue of (1) above, namely that a sequence of curves (with fixed genus) must 
have a subsequence converging to a flow graph.  One can show moreover that, if the curves
have one positive puncture, then the limiting graph must have the same genus as the curves.  
The basic point is that for genus to disappear into an edge or vertex, the corresponding `piece' 
of a rescaled limiting curve must be mapping by a multiple cover; hence the whole curve must be.
But this cannot happen if there is exactly one positive puncture.  In particular, when studying 
curves with one positive puncture, if all limiting flow graphs are rigid flow trees, then all one-positive-puncture 
curves for nearby $\Lambda_t$ are disks.  This fact was already used 
in the proofs of \cite[Lemma 4.16]{ENg}, and 
in the flow-tree versions of the proofs of \cite[Prop. 2.1, 2.1]{unknot}; the above argument was communicated to us by Ekholm.

\section{Chords, trees, and chains } 

Here we prove various geometric results needed to apply Lemma \ref{recursion} to $\Lambda_\Gamma$ and determine the resulting $A_\rho$. 
First let us recall the construction of $\Lambda_\Gamma$ from \cite{Treumann-Zaslow}.  
Recall $\Gamma$ is a trivalent graph in $S^2$.  Choose a non-negative function $f_\Gamma$ such that 
$f_\Gamma$ has a single maximum in every face, vanishes along the edges, and so that 
$$\Phi_\Gamma := Graph(f_\Gamma) \cup Graph(-f_\Gamma)$$
is the front projection of a smooth Legendrian, which we denote $\Lambda'_\Gamma \subset J^1 S^2$.  
(Requiring smoothness of $\Lambda'_\Gamma$ implies in particular that $\Phi_\Gamma$ will have the $D_4^-$ singularity of 
Arnol'd \cite{AGV-I} at the vertices of $\Gamma$.) 
Then $\Lambda_\Gamma$ is obtained by implanting $\Lambda'_\Gamma$ in a standard
neighborhood of the standard unknot in the standard contact $S^5$.  

Let $F \subset S^2$ be a face of $\Gamma$.  We write $F_+$ and $F_-$ for the upper and lower sheets
of the preimage of $F$ under the projection $\pi: \Lambda_\Gamma \to S^2$.  These subsets of $\Lambda_\Gamma$ 
meet only at the vertices, as depicted in Figure \ref{face preimages}.

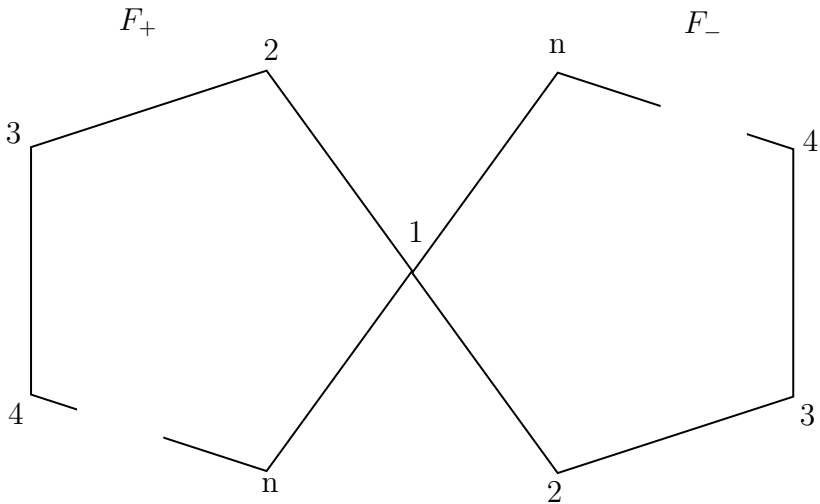
\begin{figure}[h]
\begin{center}
\tikzset{every picture/.style={line width=0.75pt}} 

\begin{tikzpicture}[x=0.75pt,y=0.75pt,yscale=-1,xscale=1]

\draw   (236.5,153.25) -- (163.08,254.3) -- (44.29,215.7) -- (44.29,90.8) -- (163.08,52.2) -- cycle ;
\draw  [color={rgb, 255:red, 255; green, 255; blue, 255 }  ,draw opacity=1 ][fill={rgb, 255:red, 255; green, 255; blue, 255 }  ,fill opacity=1 ] (68,215) -- (110.5,215) -- (110.5,257) -- (68,257) -- cycle ;
\draw   (236.5,154.25) -- (309.93,53.21) -- (428.72,91.83) -- (428.7,216.74) -- (309.9,255.31) -- cycle ;
\draw  [color={rgb, 255:red, 255; green, 255; blue, 255 }  ,draw opacity=1 ][fill={rgb, 255:red, 255; green, 255; blue, 255 }  ,fill opacity=1 ] (404.8,92.21) -- (362.3,92.21) -- (362.31,50.21) -- (404.81,50.21) -- cycle ;

\draw (233,127) node [anchor=north west][inner sep=0.75pt]   [align=left] {1};
\draw (160,35) node [anchor=north west][inner sep=0.75pt]   [align=left] {2};
\draw (30,77) node [anchor=north west][inner sep=0.75pt]   [align=left] {3};
\draw (31.29,218.7) node [anchor=north west][inner sep=0.75pt]   [align=left] {4};
\draw (159.08,256.3) node [anchor=north west][inner sep=0.75pt]   [align=left] {n};
\draw (303,258) node [anchor=north west][inner sep=0.75pt]   [align=left] {2};
\draw (430.49,219.42) node [anchor=north west][inner sep=0.75pt]   [align=left] {3};
\draw (432,81) node [anchor=north west][inner sep=0.75pt]   [align=left] {4};
\draw (304,35) node [anchor=north west][inner sep=0.75pt]   [align=left] {n};
\draw (87,19) node [anchor=north west][inner sep=0.75pt]   [align=left] {$\displaystyle F_{+}$};
\draw (372,21) node [anchor=north west][inner sep=0.75pt]   [align=left] {$\displaystyle F_{-}$};
\end{tikzpicture}
\end{center}
\caption{\label{face preimages} Preimages $F_+, F_- \subset \Lambda_\Gamma$ of a face $F\subset S^2$ of $\Gamma$. 
Vertices with the same number are identified.}
\end{figure}

\begin{lemma} \label{positive}
For any trivalent graph $\Gamma$, the pair $(S^5, \Lambda_\Gamma)$ is Reeb positive.  The index one Reeb chords
are in bijection with the faces of $\Gamma$.  We may ensure that any particular one of these chords is the lowest
action chord. 
\end{lemma} 
\begin{proof} 
The standard unknot $S^2 \subset S^5$ has a unique Reeb chord, which is of index 2.  For the satellite $\Lambda_\Gamma$, 
using Lemma \ref{satellite positivity} it suffices to check in $J^1 S^2$, where the Reeb chords correspond to the maxima of the function $f_\Gamma$
and are readily seen to have index 1. 
\end{proof}

\begin{proposition} \label{flow tree count}  
Fix a trivalent graph $\Gamma \subset S^2$, face $F$ of $\Gamma$, 
and corresponding index one Reeb chord $\rho_F$ of $\Lambda_\Gamma$. 
Then (for an appropriate complex structure) the 
holomorphic curves in $(\R \times S^5, \R \times \Lambda_\Gamma)$ with 
one positive puncture along $\rho_F$ are all transversely cut out disks,  in bijective correspondence with vertices $v$ of $F$. 

We may arrange that, at the Reeb chord $\rho_F$, tangents to disk boundaries lie in some half-plane.  
The boundaries themselves are isotopic to the paths depicted in Figure \ref{boundaries}: 
\end{proposition}

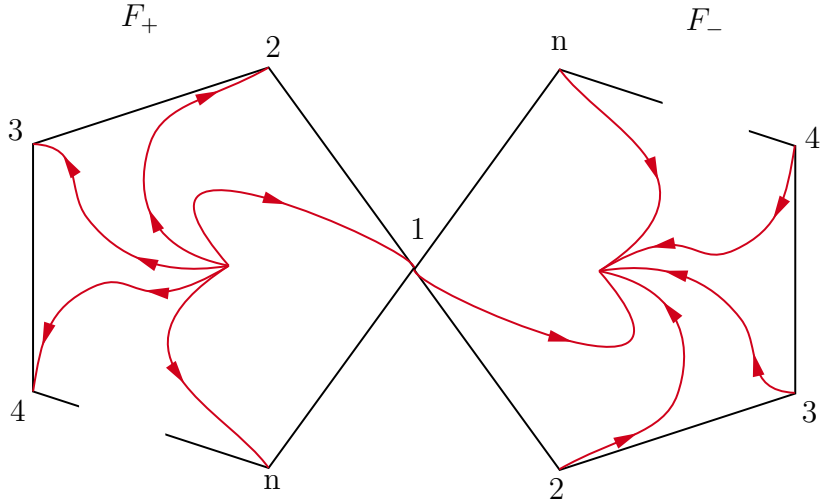
\begin{figure}[h]

\begin{center}
\tikzset{every picture/.style={line width=0.75pt}} 

\begin{tikzpicture}[x=0.75pt,y=0.75pt,yscale=-1,xscale=1]

\draw   (236.5,153.25) -- (163.08,254.3) -- (44.29,215.7) -- (44.29,90.8) -- (163.08,52.2) -- cycle ;
\draw [color={rgb, 255:red, 208; green, 2; blue, 27 }  ,draw opacity=1 ]   (143,152.25) .. controls (72.5,71) and (236.5,143.25) .. (236.5,153.25) ;
\draw [shift={(172.18,121.26)}, rotate = 197.5] [fill={rgb, 255:red, 208; green, 2; blue, 27 }  ,fill opacity=1 ][line width=0.08]  [draw opacity=0] (12,-3) -- (0,0) -- (12,3) -- cycle    ;
\draw [color={rgb, 255:red, 208; green, 2; blue, 27 }  ,draw opacity=1 ]   (163.08,52.2) .. controls (135.5,68) and (110.5,70) .. (103.5,90) .. controls (96.5,110) and (96.5,143) .. (143,152.25) ;
\draw [shift={(138.27,64.02)}, rotate = 157.17] [fill={rgb, 255:red, 208; green, 2; blue, 27 }  ,fill opacity=1 ][line width=0.08]  [draw opacity=0] (12,-3) -- (0,0) -- (12,3) -- cycle    ;
\draw [shift={(102.63,123.72)}, rotate = 65.32] [fill={rgb, 255:red, 208; green, 2; blue, 27 }  ,fill opacity=1 ][line width=0.08]  [draw opacity=0] (12,-3) -- (0,0) -- (12,3) -- cycle    ;
\draw [color={rgb, 255:red, 208; green, 2; blue, 27 }  ,draw opacity=1 ]   (44.29,90.8) .. controls (70.5,92) and (60.73,113.93) .. (71.5,128) .. controls (82.27,142.07) and (102.5,159) .. (143,152.25) ;
\draw [shift={(59.39,96.38)}, rotate = 57.7] [fill={rgb, 255:red, 208; green, 2; blue, 27 }  ,fill opacity=1 ][line width=0.08]  [draw opacity=0] (12,-3) -- (0,0) -- (12,3) -- cycle    ;
\draw [shift={(95.42,147.62)}, rotate = 22.19] [fill={rgb, 255:red, 208; green, 2; blue, 27 }  ,fill opacity=1 ][line width=0.08]  [draw opacity=0] (12,-3) -- (0,0) -- (12,3) -- cycle    ;
\draw [color={rgb, 255:red, 208; green, 2; blue, 27 }  ,draw opacity=1 ]   (44.29,215.7) .. controls (48.56,189.22) and (51.79,173.56) .. (74.5,163) .. controls (97.21,152.44) and (95.5,181) .. (143,152.25) ;
\draw [shift={(48.92,192.81)}, rotate = 288.56] [fill={rgb, 255:red, 208; green, 2; blue, 27 }  ,fill opacity=1 ][line width=0.08]  [draw opacity=0] (12,-3) -- (0,0) -- (12,3) -- cycle    ;
\draw [shift={(100.73,164.76)}, rotate = 7.08] [fill={rgb, 255:red, 208; green, 2; blue, 27 }  ,fill opacity=1 ][line width=0.08]  [draw opacity=0] (12,-3) -- (0,0) -- (12,3) -- cycle    ;
\draw [color={rgb, 255:red, 208; green, 2; blue, 27 }  ,draw opacity=1 ]   (143,152.25) .. controls (72.5,198) and (142.5,225) .. (163.08,254.3) ;
\draw [shift={(120.66,211.85)}, rotate = 235.13] [fill={rgb, 255:red, 208; green, 2; blue, 27 }  ,fill opacity=1 ][line width=0.08]  [draw opacity=0] (12,-3) -- (0,0) -- (12,3) -- cycle    ;
\draw  [color={rgb, 255:red, 255; green, 255; blue, 255 }  ,draw opacity=1 ][fill={rgb, 255:red, 255; green, 255; blue, 255 }  ,fill opacity=1 ] (68,215) -- (110.5,215) -- (110.5,257) -- (68,257) -- cycle ;
\draw   (236.5,154.25) -- (309.93,53.21) -- (428.72,91.83) -- (428.7,216.74) -- (309.9,255.31) -- cycle ;
\draw [color={rgb, 255:red, 208; green, 2; blue, 27 }  ,draw opacity=1 ]   (329.79,154.95) .. controls (400.28,236.21) and (236.5,164.25) .. (236.5,154.25) ;
\draw [shift={(316.18,190.55)}, rotate = 194.19] [fill={rgb, 255:red, 208; green, 2; blue, 27 }  ,fill opacity=1 ][line width=0.08]  [draw opacity=0] (12,-3) -- (0,0) -- (12,3) -- cycle    ;
\draw [color={rgb, 255:red, 208; green, 2; blue, 27 }  ,draw opacity=1 ]   (309.69,255) .. controls (337.28,239.2) and (362.28,237.21) .. (369.28,217.21) .. controls (376.28,197.21) and (376.29,164.21) .. (329.79,154.95) ;
\draw [shift={(349.08,236.78)}, rotate = 155.13] [fill={rgb, 255:red, 208; green, 2; blue, 27 }  ,fill opacity=1 ][line width=0.08]  [draw opacity=0] (12,-3) -- (0,0) -- (12,3) -- cycle    ;
\draw [shift={(361.75,170.24)}, rotate = 49.93] [fill={rgb, 255:red, 208; green, 2; blue, 27 }  ,fill opacity=1 ][line width=0.08]  [draw opacity=0] (12,-3) -- (0,0) -- (12,3) -- cycle    ;
\draw [color={rgb, 255:red, 208; green, 2; blue, 27 }  ,draw opacity=1 ]   (428.49,216.42) .. controls (402.28,215.21) and (412.06,193.28) .. (401.29,179.21) .. controls (390.52,165.14) and (370.29,148.21) .. (329.79,154.95) ;
\draw [shift={(407.57,196.13)}, rotate = 77.08] [fill={rgb, 255:red, 208; green, 2; blue, 27 }  ,fill opacity=1 ][line width=0.08]  [draw opacity=0] (12,-3) -- (0,0) -- (12,3) -- cycle    ;
\draw [shift={(362.37,154.8)}, rotate = 13.37] [fill={rgb, 255:red, 208; green, 2; blue, 27 }  ,fill opacity=1 ][line width=0.08]  [draw opacity=0] (12,-3) -- (0,0) -- (12,3) -- cycle    ;
\draw [color={rgb, 255:red, 208; green, 2; blue, 27 }  ,draw opacity=1 ]   (428.51,91.51) .. controls (424.24,118) and (421,133.65) .. (398.29,144.21) .. controls (375.58,154.77) and (377.3,126.21) .. (329.79,154.95) ;
\draw [shift={(417.51,128.73)}, rotate = 299.64] [fill={rgb, 255:red, 208; green, 2; blue, 27 }  ,fill opacity=1 ][line width=0.08]  [draw opacity=0] (12,-3) -- (0,0) -- (12,3) -- cycle    ;
\draw [shift={(356.08,142.93)}, rotate = 349.79] [fill={rgb, 255:red, 208; green, 2; blue, 27 }  ,fill opacity=1 ][line width=0.08]  [draw opacity=0] (12,-3) -- (0,0) -- (12,3) -- cycle    ;
\draw [color={rgb, 255:red, 208; green, 2; blue, 27 }  ,draw opacity=1 ]   (329.79,154.95) .. controls (400.3,109.21) and (330.3,82.2) .. (309.73,52.9) ;
\draw [shift={(359.59,109.25)}, rotate = 248.94] [fill={rgb, 255:red, 208; green, 2; blue, 27 }  ,fill opacity=1 ][line width=0.08]  [draw opacity=0] (12,-3) -- (0,0) -- (12,3) -- cycle    ;
\draw  [color={rgb, 255:red, 255; green, 255; blue, 255 }  ,draw opacity=1 ][fill={rgb, 255:red, 255; green, 255; blue, 255 }  ,fill opacity=1 ] (404.8,92.21) -- (362.3,92.21) -- (362.31,50.21) -- (404.81,50.21) -- cycle ;

\draw (233,127) node [anchor=north west][inner sep=0.75pt]   [align=left] {1};
\draw (160,35) node [anchor=north west][inner sep=0.75pt]   [align=left] {2};
\draw (30,77) node [anchor=north west][inner sep=0.75pt]   [align=left] {3};
\draw (31.29,218.7) node [anchor=north west][inner sep=0.75pt]   [align=left] {4};
\draw (159.08,256.3) node [anchor=north west][inner sep=0.75pt]   [align=left] {n};
\draw (303,258) node [anchor=north west][inner sep=0.75pt]   [align=left] {2};
\draw (430.49,219.42) node [anchor=north west][inner sep=0.75pt]   [align=left] {3};
\draw (432,81) node [anchor=north west][inner sep=0.75pt]   [align=left] {4};
\draw (304,35) node [anchor=north west][inner sep=0.75pt]   [align=left] {n};
\draw (87,19) node [anchor=north west][inner sep=0.75pt]   [align=left] {$\displaystyle F_{+}$};
\draw (372,21) node [anchor=north west][inner sep=0.75pt]   [align=left] {$\displaystyle F_{-}$};
\end{tikzpicture}
\end{center}
\caption{\label{boundaries} Boundaries of disks.} 
\end{figure}

\begin{figure}[h]
    \begin{picture}(150,135)
    \put(-100,-10){\includegraphics[width=6cm, trim=0cm 2cm 8cm 0cm, clip]{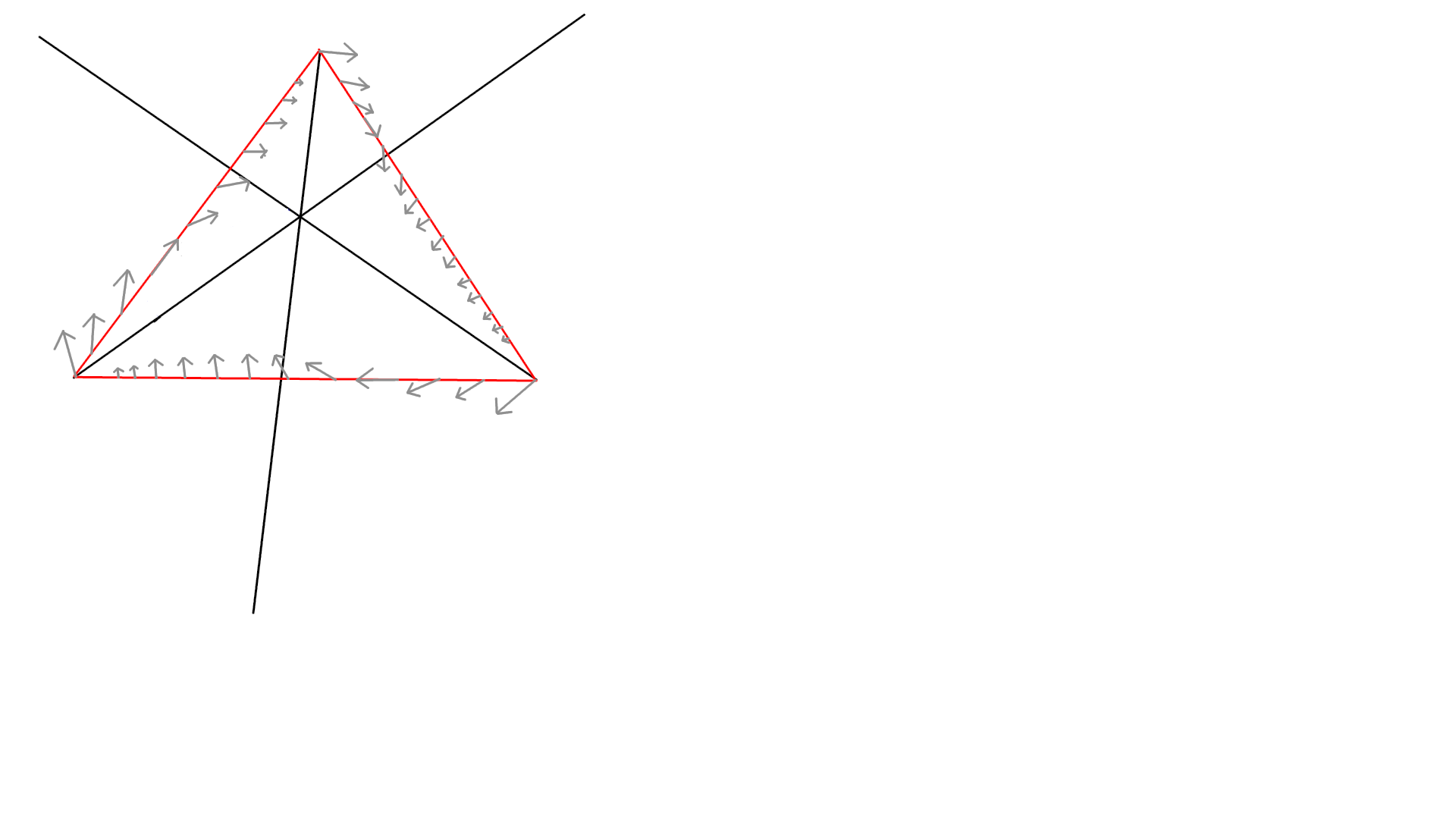}}
    \put(50,0){\includegraphics[width=6.5cm, trim=0cm 2.5cm 8cm 0cm, clip]{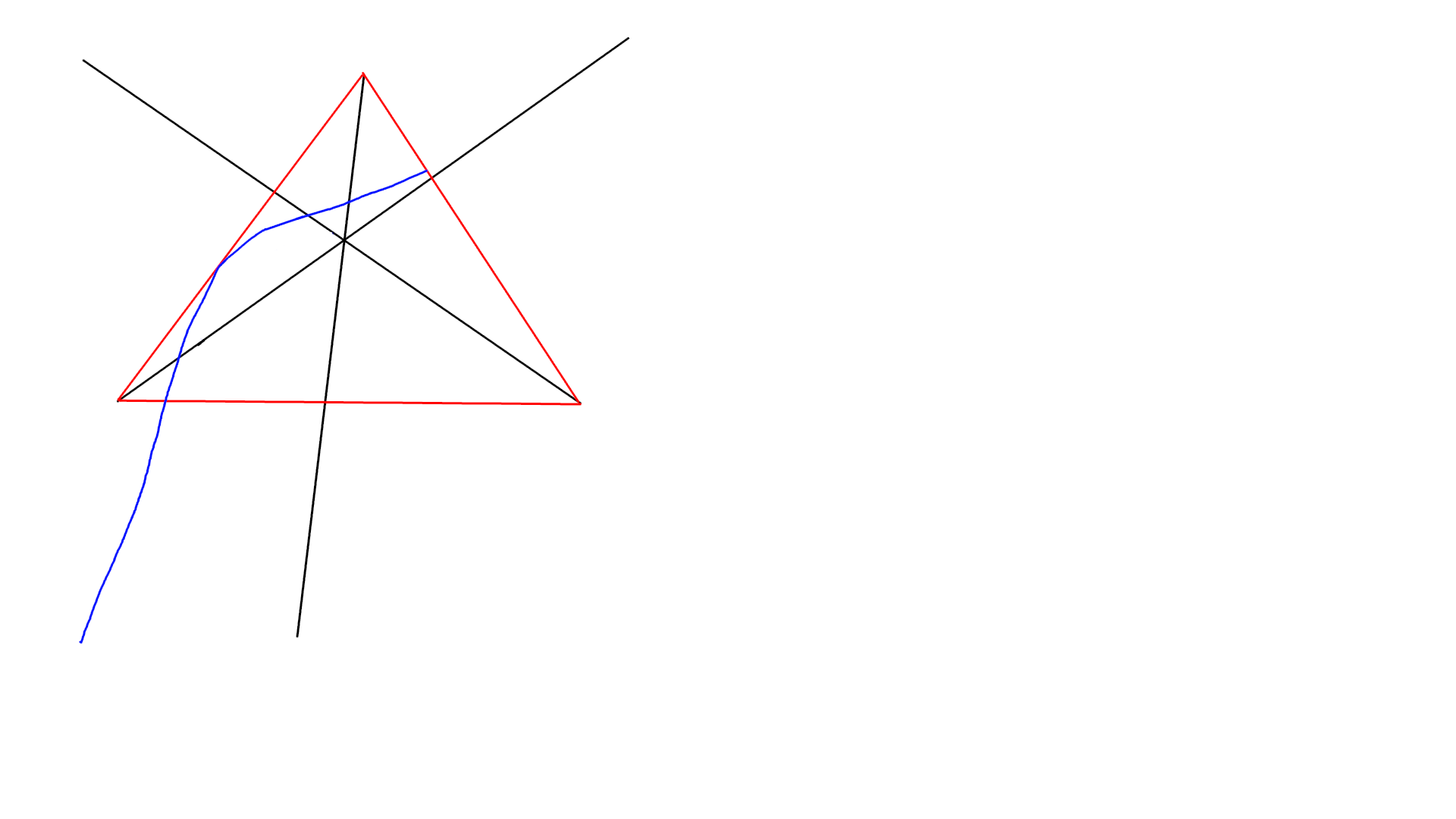}}
    \end{picture}
    \caption{On the left is the locus of singularities of a generic perturbation of the $D_4^-$ vertex to the base of the jet bundle. 
    The red lines indicate cusp-edges and where they meet there is a swallowtail singularity.  Within the red triangle, 
    the projection to the base is generically 4:1.  We locally number the sheets $(1)$ to $(4)$ from highest to lowest
    with respect to `function value' direction; note this numbering changes when crossing the black lines. 
    The grey arrows indicate the 
    gradient of the function difference on the sheets $(2)$ and $(4)$. 
    On the right, we depict the flow tree whose existence can be inferred immediately. 
    Its only vertices are a positive puncture at the Reeb chord of the face, a switch and a cusp-end.}
    \label{Generic-Perturbation}
\end{figure}

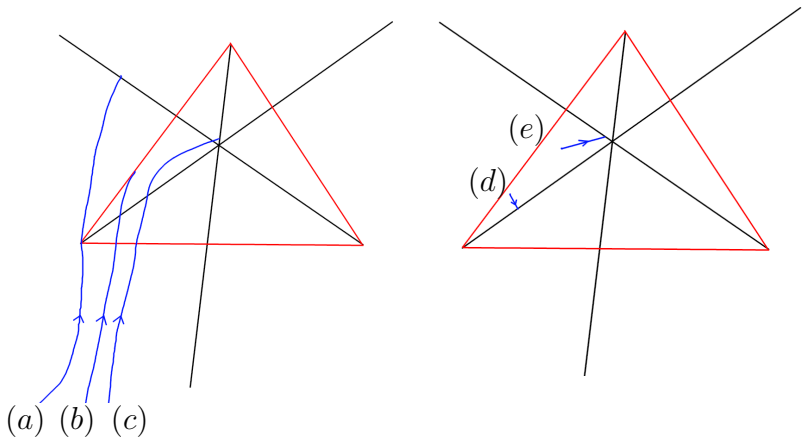
\begin{figure}[h]
    \begin{picture}(150,135)
    \put(-100,0){\includegraphics[width=6cm, trim=0cm 2cm 8cm 0cm, clip]{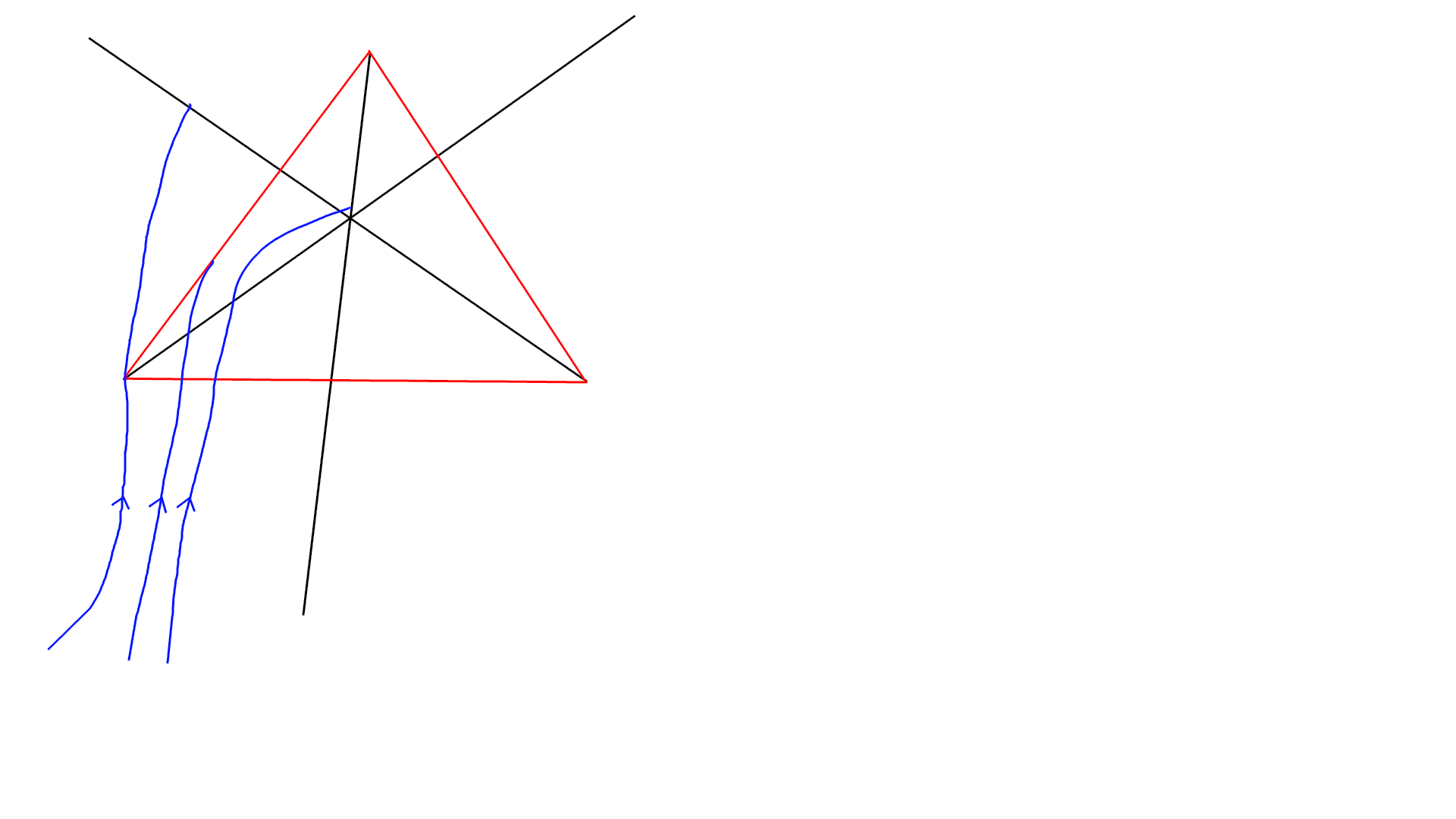}}
    \put(50,10){\includegraphics[width=6.5cm, trim=0cm 2.5cm 8cm 0cm, clip]{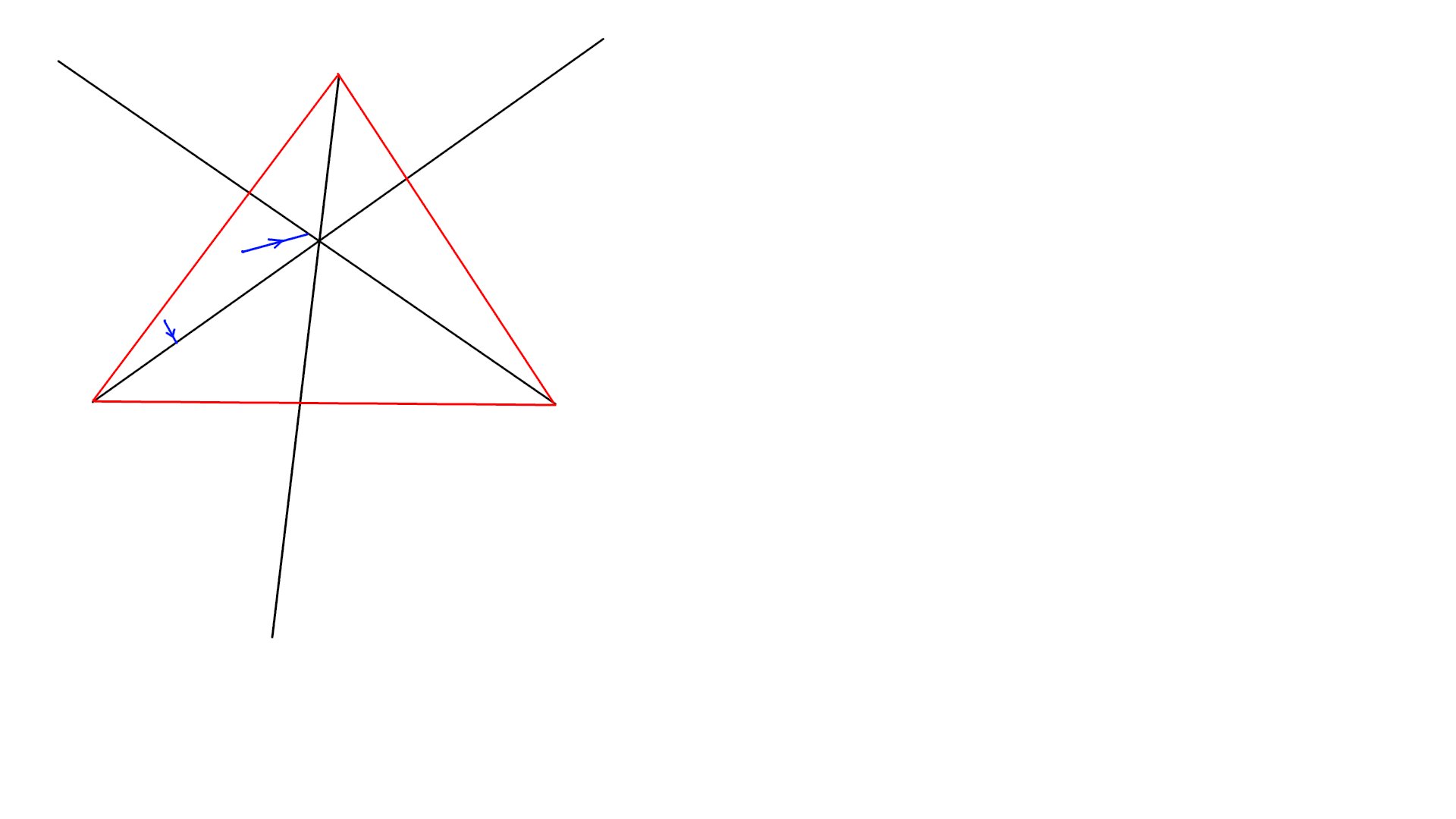}}
    \put(-100,-10){$(a)$}
    \put(-80,-10){$(b)$}
    \put(-60,-10){$(c)$}
    \put(75,80){$(d)$}
    \put(90,100){$(e)$}
    \end{picture}
    \vspace{4mm}
    \caption{Some possible flow lines close to the generic graph. Flow $(a)$ is the flow line which passes the swallowtail point, $(b)$ is the unique flow line which limits to the order $1$ tangency, $(c)$ is another flow line which necessarily misses the cusp edge, 
    $(d)$ is a flow carried by sheets $(2)$ and $(3)$, and $(e)$ is carried by sheets $(3)$ and $(4)$.} 
    \label{Possible trajectories}
\end{figure}

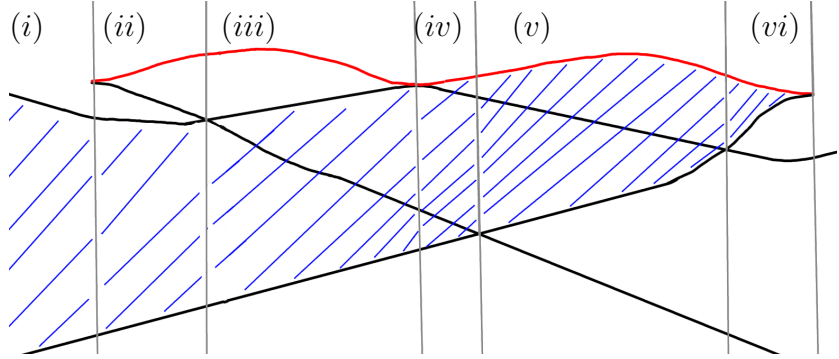
\begin{figure}[h]
    \begin{picture}(-100,125)
    \put(-200,0){\includegraphics[width=11cm, trim=2cm 1.5cm 1cm 2cm, clip]{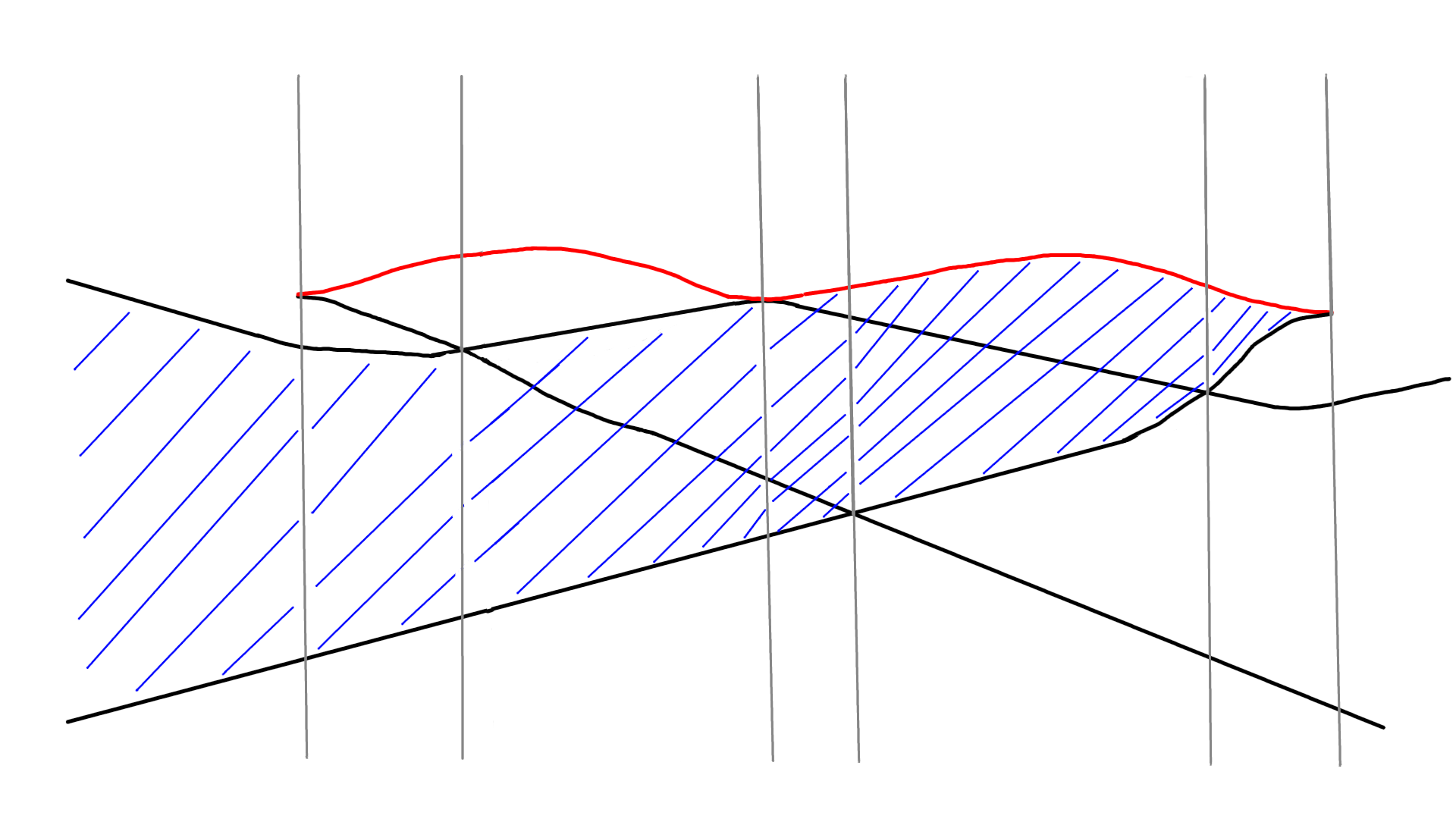}}
    \put(-200,120){$(i)$}
    \put(-165,120){$(ii)$}
    \put(-120,120){$(iii)$}
    \put(-47.5,120){$(iv)$}
    \put(-10,120){$(v)$}
    \put(80,120){$(vi)$}
    \end{picture}
    \caption{A cut of the front projection along the flow tree from Figure \ref{Generic-Perturbation}. The blue shaded area indicates the top and lower levels of the flow tree. Each of the grey lines indicates an intersection of the flow tree with one of the arcs in Figure \ref{Generic-Perturbation}. } 
    \label{Flow-tree-slice}
\end{figure}

\begin{figure}[h]
    \begin{picture}(-100,125)
    \put(-150,0){\includegraphics[width=6cm, trim=0cm 2cm 8cm 0cm, clip]{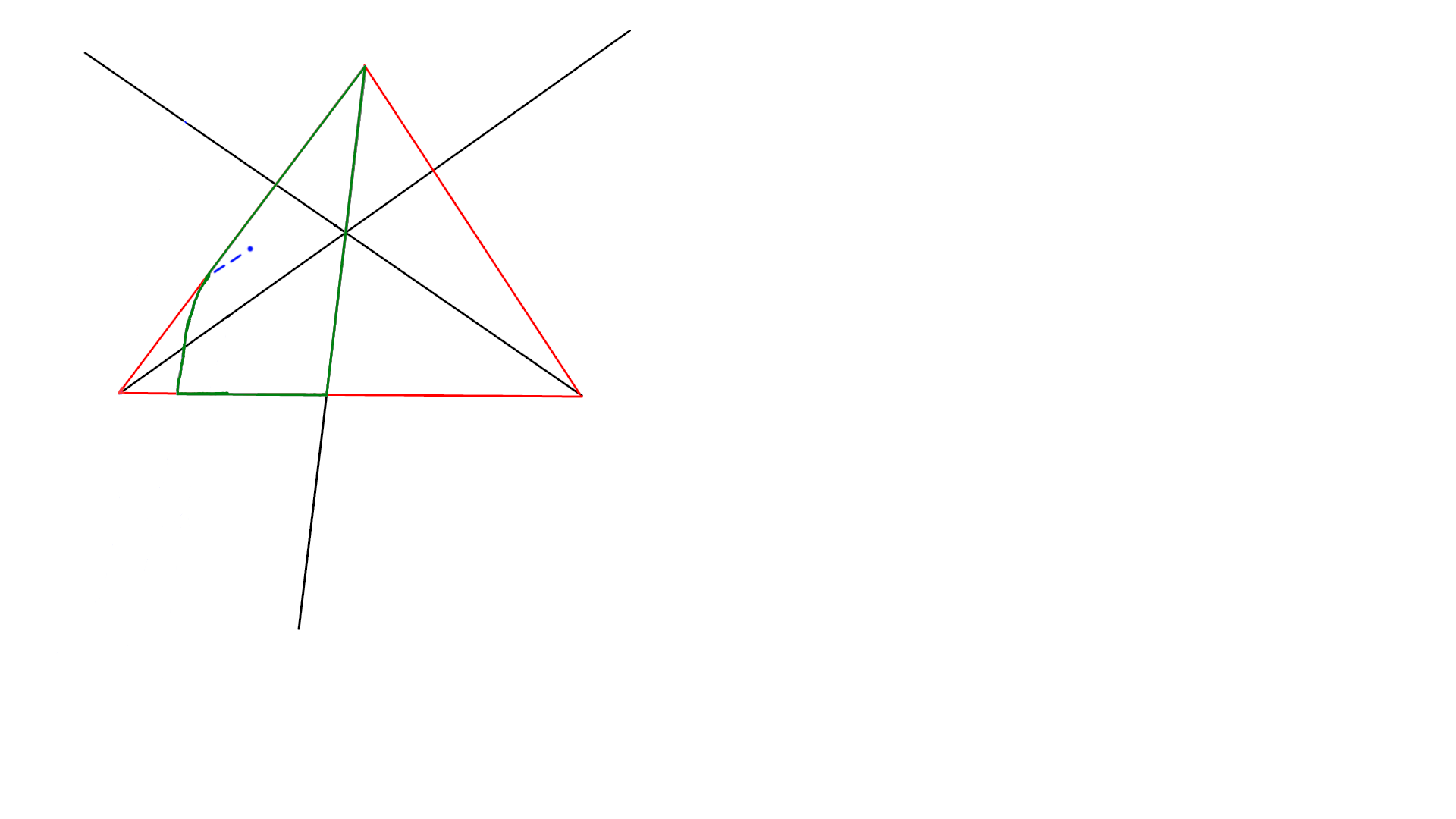}}
    \end{picture}
    \caption{A breaking of the flow tree in stage $(iv)$ which is carried by sheets $(2)$ and $(4)$ is locked into the region bounded by green arcs. The dashed blue line indicates the part of the previous flow tree while the blue point indicates the
    putative trivalent vertex.} 
    \label{Special-Case}
\end{figure}

\begin{proof}
As both $\Lambda_\Gamma$ and the index one Reeb chord in question sit inside the $J^1 S^2 \subset S^5$, 
all holomorphic curves whose only punctures lie on said chords 
must lie in the symplectization of $J^1 S^2$.

Let us consider Morse flow graphs for $\Lambda_\Gamma$. 
Fix a metric on $S^2$.  Recall that there is an index one Reeb chord
associated to each face; it is obvious that the corresponding flow graphs are just the simple paths from the 
corresponding maximum of $f_\Gamma$ 
to the vertices.  The condition on the tangents to paths is easily arranged by appropriate choice of metric. 

If we knew a version of Theorem \ref{morse flow tree theorem} which allowed for fronts with 
$D_4^-$ singularities, we could conclude immediately.  Instead, we will perturb $\Lambda_\Gamma$
by Legendrian isotopy to some $\widetilde{\Lambda}_\Gamma$, so that
in the front projection the $D_4^-$ singularities become a configuration of swallowtails,
and check
that the same flow graph configurations appear.\footnote{Similar arguments appear in e.g. \cite{EENS, Rizell}. The specific result we need here was known to Casals, and announced in \cite{Casals-Murphy}, with a proof sketched in \cite[Minute 53]{Casals-lecture}. Casals communicated a similar sketch to us directly, and we also benefitted from discussions with Ekholm around this question.}

Let us describe how to construct $\widetilde{\Lambda}_\Gamma$.  We describe it in terms of the front, 
which will be identical to that of $\Lambda_\Gamma$ away from the $D_4^-$ vertices.  Near these we choose the generic deformation of a $D_4^-$ singularity 
by ``pushing out 3 swallowtails". In particular, after a local change of the metric $g$, we may assume that there is a unique order $1$ tangency between sheets $(2)$ and $(4)$ as labeled in Figure \ref{Generic-Perturbation}.

Let us now consider flow graphs for this $\widetilde{\Lambda}_\Gamma$.  Put the one positive puncture
at the Reeb chord for some face.  The flow from here cannot go to an edge of the graph, since the `function
difference'  \eqref{function difference} vanishes there; thus the flow must go to a neighborhood of some vertex.  

Fix a face and a vertex of that face.  We first consider Morse flow trees without trivalent vertices beginning at 
the Reeb chord in the face and ending near the vertex.  We claim there is exactly one. 
It is depicted in Figure \ref{Generic-Perturbation}-right; in the flow tree jargon, it has one switch 
and one cusp-end. The uniqueness of this tree
 follows directly from existence and uniqueness of trajectories of vector fields 
 (nearby trajectories depicted in Figure \ref{Possible trajectories}-left),
 and the fact that,
 for rigid trees and generic metrics, 
 switches may only appear at order $1$ tangencies.

To complete the argument, we have to eliminate the possibility of trivalent vertices from appearing.  There are two
kinds of these; in the usual jargon they are termed $Y_0$ and $Y_1$ vertices.  It will suffice to show 
that along the path of our existing tree, no branching can occur. 

A $Y_1$ vertex can only occur at a point in the base above which, in the front projection, 
a cusp edge sits between two other sheets.  In our situation, the cusp-edges are between the uppermost sheets 
so can never be involved in a $Y_1$ vertex.

The $Y_0$ vertex requires a sheet to appear between the two sheets presently carrying the flow tree. 
We divide our existing tree into six stages, $(i)$ through $(vi)$, as depicted in 
 Figure \ref{Flow-tree-slice}.  We will exclude any $Y_0$ vertex appearing during any of these stages. 
 We can immediately exclude stages $(i)$, $(ii)$ and $(vi)$, 
 since a $Y_0$ vertex may only appear if there is a sheet in between the ones which carry the flow tree.

If a $Y_0$ vertex appears in stage $(iii)$ then one of the outgoing edges must be 
carried by sheets $(2)$ and $(3)$; the resulting flow line
must tend to the intersection of those sheets, see trajectory $(d)$ in Figure \ref{Possible trajectories}; contradiction. The same is true for stage $(v)$. 

In stage $(iv)$, we must consider two different possible breakings, 
one which has an outgoing edge carried by sheets $(2)$ and $(4)$, and one which has an outgoing edge carried by sheets $(3)$ and $(4)$.  The case of sheets $(3)$ and $(4)$ is illustrated by trajectory  $(E)$ in Figure \ref{Possible trajectories},
which can only limit towards the intersection of sheets $(3)$ and $(4)$; contradiction.

The other case is illustrated in Figure \ref{Special-Case}: A flow tree which is carried by sheets $(2)$ and $(4)$ in this area is locked into the green area. By uniqueness of flows it may not touch the original part of the flow tree and by construction it cannot come close to any of the cusp-edges. Since there are no Reeb chords in this area is must limit towards the intersection of those two sheets; contradiction.  This completes the argument. 
\end{proof}

Having asked that all tangents to the flow trees at the Reeb chord lie in one half plane distinguishes a connected
component of the complement of the paths, namely the component which contains said half plane in $F_-$ and its negative
in $F_+$.  We term this the {\em distinguished component}.  Additionally,  we may use this to fix a numbering of the vertices. 
Consider in $F_+$ the arcs bounding holomorphic curves.  They are traveling
outward from the point where the Reeb line hits $F_+$.  Per
Proposition \ref{flow tree count},   
Consider the arc which, within this half plane, is the first met when traveling counterclockwise.  
We distinguish the vertex to which this arc travels as $v = v_1$, then continue numbering the vertices counterclockwise from
this one.  Then the distinguished component is the triangle whose vertices are the endpoint of the Reeb chord, 
 $v_1$, and $v_n$.

\begin{construction} \label{4 chain}
There is an $\R$-invariant nonwhere vanishing vector field $v$ on $\R \times \Lambda_\Gamma$ and
 $\R$-invariant 4-chain $V$ in $\R \times S^5$ such that 
$\partial V = 2 (\R \times \Lambda_\Gamma)$, and locally near $\partial V$, 
we have $V = \pm J\cdot v$. In addition, 
$$V^{\circ} \cap (\R \times \Lambda_\Gamma) = \R \times \mathfrak{p}_\Gamma$$
where $\mathfrak{p}_\Gamma$ lies over the distinguished component.

Moreover, the curves from Proposition \ref{flow tree count}
are disjoint from $V$, and framing their boundaries using $v$ gives the same result
as framing using the positive $\R$ direction.
\end{construction}
\begin{proof}
We are interested in constructing
 $\R$-invariant structures in $S^5 \times \R$, so will first build a corresponding 3-chain $U$ in $S^5$.  

Consider the Reeb vector field $R$ and the coordinate vector field $\frac{d}{dr}$ of $S^5\times \R_r$. 
Recall that the complex structure $J$ is chosen to ensure
$R = J \frac{d}{dr}$.   We will write $e^{Rt}$ for the time $t$ flow of the Reeb vector field.   

We begin by constructing $U$ locally inside $J^1 S^2 \times \R_r $.
First consider $3$-chains $U_0^+$ and $U_0^-$ defined by following $\Lambda_\Gamma$
for a long time along the positive or negative Reeb flow, and orient them so that 
$$\partial U_0^+ = \Lambda_\Gamma - e^{Rt}_* \Lambda_\Gamma \qquad \qquad  
\partial U_0^- = \Lambda_\Gamma - e^{-Rt}_* \Lambda_\Gamma$$

The $4$-chain $V_0^{\pm} = U_0^{\pm} \times \R_r$ is, locally near $\Lambda_\Gamma \times \R$, 
given by $J \frac{d}{dr}$.  Note that $U_0^+$ intersects $\Lambda_\Gamma$ at the positive endpoints
of the Reeb chords; similarly for $U_0^-$.  Regarding the $J$-holomorphic curves, recall that these 
may be taken arbitrarily close to the disk interpolating between the lift of the Morse flow tree to $\Lambda_\Gamma$. 
Thus $U_0^{\pm}$ meets these curves only along the Reeb chords. 

We will now adjust $U_0^{\pm}$ to avoid these coincidences along Reeb chords.   
Recall that the choice of $J$ depends ultimately on a metric on $S^2$.  For convenience, 
we take this metric to be locally the flat metric at the points where the Reeb chords appear.

Consider a local model around the Reeb chords. 
Locally the picture is given by fronts $f_+(x,y)=-x^2-y^2+C$ and $f_-(x,y)=x^2+y^2-C$
 in $J^0(\R^2_{x, y})$.  We may also assume 
 all the boundaries of the $J$-holomorphic curves are in the halfspace $\{x\geq 0\}$.
The Legendrian is locally given by:
\begin{align*}
    j^1f_+= \{(x,y,-2x,-2y,-x^2-y^2+C)\} \qquad \qquad j^1f_- = \{(x,y,2x,2y,x^2+y^2-C)\}
\end{align*}
In this trivialisation $J^1(\R^2)=\R^2_{x,y} \times\R^2_{dx,dy} \times \R_R$ consider a vector field given by 
$$v_\epsilon = \epsilon \cdot \lambda(x,y,dx,dy,R) \cdot (0,1,0,-2,-2y)$$ 
where $\lambda(\cdot)$ is a bump function in $J^1(\R^2)$ centered at $(0,0,0,0,C)$ and $\epsilon \ne 0$ is a real number.
Now consider the flow of $v'=J(\frac{d}{dr}+v_\epsilon)$ of $\Lambda_\Gamma$ in the Lagrangian projection. One immediately verifies that this vector field displaces the intersection of $\Pi_C(e_*^{v't}(j^1f_-))$ and $\Pi_\C(j^1f_+)$ along the $y$-axis. Thus by choosing $\lambda$ and $\epsilon$ small enough the intersection point is in the desired chamber of the face $F$. We do a similar push-off around the endpoints of the Reeb chords for the negative push-offs.

\vspace{2mm}
We now have a 3-chain $U_1 = U_1^+ + U_1^-$ in $J^1 S^2$ which has the asserted properties within $J^1 S^2$, but which either 
goes off to infinity in the positive and negative direction in the `function value', or has a second boundary given by a large positive
and negative pushoff of $\Lambda_\Gamma$.   We take the second viewpoint, denoting these pushoffs as $\Lambda_\Gamma^+$ 
and $\Lambda_\Gamma^-$.  That is, $\partial U_1^\pm = \Lambda_\Gamma - \Lambda_\Gamma^{\pm}$.  

We transplant the whole situation to a standard neighborhood of the standard two-dimensional 
Legendrian unknot in $\R^5 = J^1 \R^2$.  Now we continue
to build the 3-chain by following the  pushoffs, this time by the Reeb flow of this $J^1 \R^2$.  
This will carry $\Lambda_\Gamma^+$ and $\Lambda_\Gamma^-$ eventually to satellites of large positive and negative pushoffs of the 
standard unknot.  These are contained in Darboux charts disjoint from the original setup, where we may cap the resulting 3-chain off arbitrarily
to obtain chains $U^+_3$ and $U^-_3$ such that $\partial U_3^\pm = \Lambda_\Gamma$.  However, 
we have created some new intersections: $(U_3^+)^\circ \cap \Lambda_\Gamma$ is four points, namely the positive ends
of the long index two Reeb chords, which live near the end of the unique Reeb chord of the standard unknot.  Similarly for $U_3^-$ and the negative
ends of these chords. 

We have not yet fixed any specifics of how we draw the graph $\Gamma$ on the standard unknot $S^2 \subset J^1 \R^2$.  
We do so now: arrange that all faces of the graph save one live in some small neighborhood of some point
near the cusp edge of the front projection, 
and the remaining face occupies all the remainder of the unknot $S^2$.   The point is to ensure that
all endpoints of the long Reeb chords live above the same face $F$.  We may moreover arrange
that all endpoints of the long Reeb chord appear in e.g. the distinguished component. \footnote{The geometric choice of the position of the index $2$ Reeb chords mimics the set-up of the Casals-Murphy CPG differential algebra \cite{Casals-Murphy}.} 

Note that each of $F^+$ and $F^-$ has two negative and two positive endpoints of long Reeb chords. 
(Each of $F^+$ and $F^-$ has one long self chord, and one chord to the other face of each orientation.) 
Choose arbitrarily inside $F^+$ and $F^-$ paths from positive endpoints to negative endpoints; there are four
such paths.  
Consider the 4-chain obtained from the capping path by taking a normal disk at each point.  
So we have these
disk bundles $\Delta, \Delta', \Delta'', \Delta'''$.  Their boundaries
are 3-chains consisting of normal disks at the endpoints of the paths, plus
a tube connecting these and disjoint from $\Lambda_\Gamma$.   We may arrange that the normal disks
at endpoints agree set-theoretically with the piece of $(U_3^\pm)^\circ$ there; they naturally
come with the opposed orientation.   So we finally set
$$U = U_3^+ + U_3^- + \partial \Delta + \partial \Delta' + \partial \Delta'' + \partial \Delta'''$$ 
This chain finally has the desired properties.  
\end{proof}

\begin{remark}
One can get away with a less careful treatment of the long Reeb chords: put the graph arbitrarily on $S^2$ 
and choose paths arbitrarily, and perform the same construction.  The resulting 4-chain will not be disjoint
from disks, but since the extra terms $\partial \Delta \times \R$ are boundaries, its interior intersection with disks
will still vanish.  This would be enough for our purposes. 
\end{remark}

\section{Some fiddling with spin structures and signs} \label{sec: spin sign}

The system to coherently orient moduli of maps from 
holomorphic curves with Lagrangian boundary conditions involves the 
choice of spin structures on the Lagrangians \cite{FOOO-II, EES-orientation},
or more generally twisted spin structures with respect to some background class.  
The spin structures affect the sign with which holomorphic curves are counted. 
In this section we construct a twisted spin structure on $\Lambda_\Sigma$ which will
later turn out to give all positive signs, and relatedly we explain how background
classes give rise to sign lines. 

A spin structure on an orientable manifold determines a spin structure on 
an orientable codimension one hypersurface; in particular, a spin structure on a surface
determines a spin structure on any embedded (or indeed, as a trivialization is local, immersed) 
closed curve in the surface.  The circle admits two spin structures, one termed 
the `Lie group spin structure' as it corresponds to the trivialization of the tangent bundle from
the Lie group structure, and the other often called the `bounding spin structure', because it is 
the restriction of the unique spin structure on a disk.  

The following notation will be convenient: 

\begin{definition} \label{spin sign}
For a simple closed curve $C$ on a surface $\Sigma$ carrying a spin structure $\sigma$, 
we will write 
$$\sigma(C) = \begin{cases} 0 & \mbox{if} \,\, \sigma|_{C} \,\, \mbox{is the bounding spin structure} \\
1 &\mbox{if}\,\, \sigma|_C  \,\, \mbox{is the Lie group spin structure} \end{cases}
$$
and
$$[C]_{\sigma} := (-1)^{1+ \sigma(C)} [C] \in Sk(\Sigma)$$
\end{definition} 

There is a classical identification between spin structures on surfaces 
and quadratic refinements of the intersection pairing \cite{Johnson-spin},
namely that such $\sigma$ as above is the restriction to simple closed curves
of a function 
$\sigma: H_1(\Sigma, \Z/2\Z) \to \Z/2\Z$ 
satisfying
\begin{equation} \label{quadratic} \sigma(C+D) = \sigma(C) + \sigma(D) + C\cdot D \end{equation} 
In fact, isomorphism classes of spin structures are in bijection with such quadratic refinements.  
Note that on a $\Z/2\Z$ vector space given skew form, 
quadratic refinements are freely and uniquely characterized by their values on a basis.

\begin{lemma} \label{spin l k} 
All the $\ell_i$ have the Lie group spin structure iff the same holds for all the $\widetilde{E}_j$. 
\end{lemma}
\begin{proof}
The claim is equivalent to the formula
$1 + \sigma(\ell_k) = \sum_{i=1}^{k-1} 1+\sigma(\widetilde{E}_i)$.  Since $\ell_2 = \widetilde{E}_1$, the 
case  $k=2$ is a tautology. 
Assume the result for $k$.  In homology, we have 
$\ell_{k+1} = \ell_k + \widetilde{E_{k}}$ and 
$\langle \ell_{k}, \widetilde{E}_{k} \rangle = 1 \pmod 2$.  
So, from Eq. \eqref{quadratic}, we conclude 
$\sigma(\ell_{k+1}) = \sigma(\ell_k) + \sigma(\widetilde{E_{k}}) + 1$. 
The result follows by induction. 
\end{proof} 

Consider an exact sequence of $\Z/2\Z$ vector spaces, 
$0 \to K \to Y \to Z \to 0$.  
Suppose $Z$
carries a skew form.  Then quadratic refinements of the pullback
form on $Y$ are all linear on $K$, and descend to a quadratic refinement
of the original form on $Z$ iff they vanish on $K$.  

\begin{lemma} \label{best spin structure} 
The surface $\Lambda_\Gamma$ admits a spin structure for which all $\widetilde{E}_i$ all carry the 
Lie group spin structure if and only if it has odd genus.   In this case, 
said spin structure is unique.
\end{lemma} 
\begin{proof}
As the $\widetilde{E}_i$ generate  $H_1(\Lambda_\Gamma, \Z/2\Z)$,
there is at most one such spin structure.  

Let $\mathbf{E}$ be the $\Z/2\Z$ vector space generated by the edges of $\Gamma$, and with
skew form pulled back from  $H_1(\Lambda_\Gamma, \Z/2\Z)$ along the map $E \mapsto \widetilde{E}$. 
Consider the unique quadratic refinement $\bar \sigma$ of the pairing on $\mathbf{E}$
satisfying $\bar \sigma(E) = 1$.  We must check whether 
$\bar{\sigma}$ vanishes on the kernel of  $\mathbf{E} \to H_1(\Lambda_\Gamma, \Z/2\Z)$. 
As $\bar \sigma$ is linear here, we may check on a basis. 

The kernel is generated by:  $R_\Gamma$, the sum of all edges in the graph, 
and, for each face $F$, the sum  $R_F$ all edges around $F$. 
Using Eq. \eqref{quadratic} we compute 
\begin{equation} \label{face spin}
\bar\sigma(R_F) \equiv \sum_{E \subset \partial F} 1 + \sum_{E \ne E' \subset \partial F} \widetilde{E} \cdot \widetilde{E'} = \#\mathrm{edges}(F) + \mathrm{vertices}(F) \equiv 0 \pmod 2
\end{equation}
\begin{equation}
\label{total spin} 
\bar \sigma(R_\Gamma) \equiv \sum_E 1 + \sum_{E \ne E'} \widetilde{E} \cdot \widetilde{E}' = \# \mathrm{edges} + 3 \# \mathrm{vertices} 
\equiv 1 + \mathrm{genus}(\Lambda_\Gamma) \pmod 2
\end{equation} 
This completes the proof.
\end{proof} 

\begin{example}
A sphere has a unique spin structure, which evidently restricts to every simple closed curve as the bounding spin structure. 
The torus has its own Lie group spin structure, which restricts to every simple closed curve as the Lie group spin structure. 
\end{example}

As we saw in Section \ref{sec: face}, the conjectures of \cite{Schrader-Shen-Zaslow} are naturally formulated
in a skein with sign and framing lines.  Framing lines naturally generated by 4-chains
in the skein-valued curve counting of \cite{SOB}, and  in Construction \ref{4 chain} we gave a 4-chain which introduces
 the framing lines which appeared in Section \ref{sec: face}.  We now explain how sign lines arise geometrically.  

Recall that, in the process of setting up the coherent orientations on moduli, 
one can fix a nontrivial `background class' $\beta \in H^2(X, \Z/2 \Z)$,
in which case, rather than a spin structure on Lagrangians $L\subset X$, one needs rather a `relative pin structure'
with respect to $\beta$.  

Represent $\beta$ by a  Poincar\'e dual closed (Borel-Moore) chain $B \in Z_{\mathrm{dim}(X)-2}^{BM}(X, \Z / 2 \Z)$; 
assume $B$ meets $L$ transversely.  Let $C$ be the boundary of a small disk in $L$ which is transverse to $L \cap B$; 
then relative pin structures for $\beta$  can be identified with spin structures on $L \setminus B$ which restrict to
the Lie group spin structure on $C$.  We will call these $B$-twisted spin structures on $L$.  
If $[B] = 0 \in H^{BM}(M, \Z/2\Z)$ then these may be identified with usual spin structures, but the identification will
depend on a choice of bounding chain for $B$. 

With the background class thus identified geometrically, 
the sign with which a holomorphic curve is counted 
will change in a family where its boundary passes once transversely $L \cap B$; 
compare e.g. \cite[Lemma A.5]{Abouzaid-loops}.  

\vspace{2mm}
Thus in the presence of such a background class we should count 
curves in the skein with a sign line at $L \cap B$.  
\vspace{2mm} 

We use a 4-chain satisfying $\partial V = 2 L$, which thus defines an element of 
$Z_{\mathrm{dim}(X)-2}^{BM}(X, \Z / 2 \Z)$, and may be chosen
as a background class.  In this case, 
the curve counts will take values in (a completion of) a skein with framing and sign lines, namely 
$Sk(L, V^{\circ} \cap L, (-1)^{V^{\circ} \cap L}).$

\begin{lemma} \label{best twisted spin structure} 
There is a unique $\mathfrak{p}_\Gamma$-twisted spin structure on $\Lambda_\Gamma$ for which
all the $\widetilde{E}$ have the Lie group spin structure. 
\end{lemma} 
\begin{proof}
Uniqueness again follows from the fact that the $\widetilde{E}$ generate $H_1(\Lambda_\Gamma, \Z/2\Z)$. 
Regarding existence: we are supposed to construct a spin structure on $\Lambda_\Gamma \setminus \mathfrak{p}_\Gamma$ 
such that all the $\widetilde{E}$ and all the $\partial F^+$ and $\partial F^-$ have the Lie group spin structure. 
Now $H_1(\Lambda_\Gamma \setminus \mathfrak{p}_\Gamma, \Z/2\Z)$ is generated by the $\widetilde{E}$
and the $\partial F^+$ subject to the sole relation $R_\Gamma = \sum_F \partial F^+$, with notation as in
the proof of Lemma \ref{best spin structure}. 

Consider the $\Z/2\Z$ vector space
generated by the edges and faces, with pairing pulled back from $H_1(\Lambda_\Gamma \setminus \mathfrak{p}_\Gamma, \Z/2\Z)$
under the map $E \mapsto \widetilde E$ and $F \mapsto \partial F^+$.   Define $\bar \sigma$ on this space 
by $\bar \sigma(E) = 1 = \bar \sigma (F)$.  To see that $\bar \sigma$ descends to  $H_1(\Lambda_\Gamma \setminus \mathfrak{p}_\Gamma, \Z/2\Z)$, 
we must check 
that $\bar \sigma( \sum E) = \bar \sigma(\sum F)$.  We computed the left hand side in Eq. \ref{total spin} to 
be $\#\mathrm{edges} + 3 \# \mathrm{vertices}$.  The right hand side is evidently $\# \mathrm{faces}$.  
By Euler's formula for planar graphs, these are congruent modulo two. 

Having shown that $\bar \sigma$ descends, we must check that $\bar \sigma(\partial F^-) = 1$.  This
follows from the relation
 $\partial F^+ + \partial F^- = R_F$ and Eq. \eqref{face spin}. 
\end{proof}

\section{Skein valued operator equations}

\vspace{2mm}
Our main result is the following:

\begin{theorem} \label{main theorem}
Let $\Gamma$ be a trivalent planar graph on $S^2$, and 
let $(X, L)$ be any filling of $(S^5, \Lambda_\Gamma)$.  
Fix any compatible vector field and 4-chain $(w, W)$ for $L$ extending the $(v, V)$ of Construction \ref{4 chain},
and a spin structure $\sigma$ on $L$; we also write $\sigma$ for the induced
spin structure on $\partial L = \Lambda_\Gamma$.    

Let $F$ be any face of $\Gamma$, and $v \in F$ any vertex, and consider
$$A_{\Gamma, F, v} := a^{-1}[\bigcirc] 
+ [\ell_2]_{\sigma} + [\ell_3]_{\sigma} + \cdots [\ell_n]_{\sigma}  \in Sk(\Lambda_\Gamma, \mathfrak{p}_{\Gamma})$$ 
where the $\ell_k$ are as in Definition \ref{def: ell}, and the subscripts indicate signs as in Definition \ref{spin sign}. 
Let $\Psi_{X, L} \in \widehat{Sk}(L)$ be the skein-valued open curve count.   Then $A_{\Gamma, F, v} \Psi_{X, L} = 0$. 

If we define curve counts using $[W] \in H_4^{BM}(X, \Z/2\Z)$ as a background class,
and choose correspondingly a $(W^\circ \cap L)$-twisted spin structure 
on $L$,  then 
the corresponding statements hold 
with $A_{\Gamma, F, v} \in Sk(\Lambda_\Gamma, \mathfrak{p}_{\Gamma}, (-1)^{\mathfrak{p}_{\Gamma}})$.
\end{theorem}
\begin{proof}
Fix a face $F$ and let $\rho$ be the Reeb chord corresponding to that face. 
By Lemma \ref{positive}, we may appeal to Lemma \ref{recursion}, from which we learn that 
$A_\rho \Psi_{X, L} = 0$, where $A_\rho$ is the count of curves in the symplectization, 
made an element of the skein by the choice of a capping path.  We determined these curves in 
Proposition \ref{flow tree count}; it remains only to choose a capping path and determine
the corresponding element of the skein.  

We take the capping path
as depicted in blue in the Figure \ref{capping}. 
By inspection, the resulting element of the skein is $A_{\Gamma, F, v}$.   
The $a^{-1}$ comes because the corresponding path displayed
is an unknot with nontrivial framing.  (Whether it is $a$ or $a^{-1}$ depends on
a global orientation choice for $L$ when defining the skein.) 
Note if we had chosen instead the capping path still in the distinguished component but 
so that $p_{\pm}$ were on the other side of it, this 
would change the relation only by an overall factor.

It remains only to discuss the signs.  Sign conventions for holomorphic curves
with Reeb punctures depend on  choices at the Reeb chord \cite{EES-orientation, EENS}. We are only interested
in the relative sign between different disks bounding the same Reeb chord, 
so these choices are irrelevant.  We are studying disks with only one positive puncture, 
so there is no discussion to be had about the determinant varying on the space of conformal structures.
The relative signs are then governed by comparing spin
structures on the boundaries of the disks, with respect to some fixed capping path.  We use 
the capping path we have already chosen.  We have eaten most of the signs into the notation; 
the only thing remaining to observe is that the curve giving the unknot contribution has vanishing winding number, 
hence the spin structure on the plane restricts to the Lie group spin structure, hence  the 
corresponding sign is $+1$.  
\end{proof} 

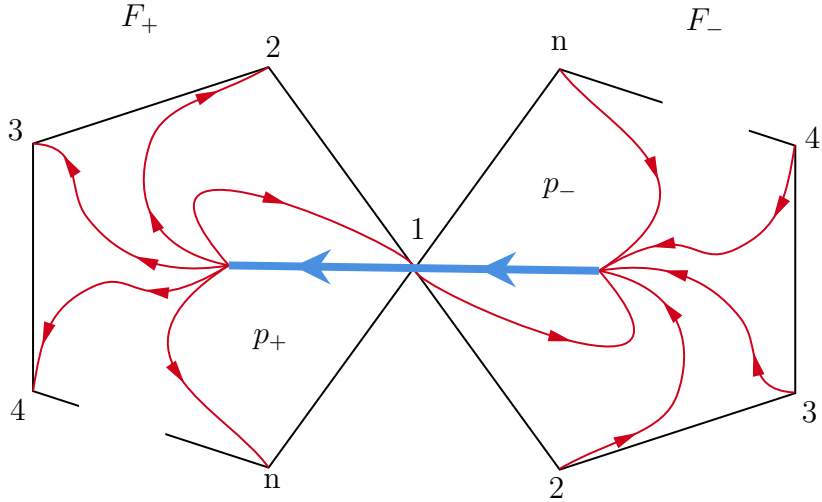
\begin{figure}
\vspace{2mm}

\begin{center}
\tikzset{every picture/.style={line width=0.75pt}} 

\begin{tikzpicture}[x=0.75pt,y=0.75pt,yscale=-1,xscale=1]

\draw   (236.5,153.25) -- (163.08,254.3) -- (44.29,215.7) -- (44.29,90.8) -- (163.08,52.2) -- cycle ;
\draw [color={rgb, 255:red, 208; green, 2; blue, 27 }  ,draw opacity=1 ]   (143,152.25) .. controls (72.5,71) and (236.5,143.25) .. (236.5,153.25) ;
\draw [shift={(172.18,121.26)}, rotate = 197.5] [fill={rgb, 255:red, 208; green, 2; blue, 27 }  ,fill opacity=1 ][line width=0.08]  [draw opacity=0] (12,-3) -- (0,0) -- (12,3) -- cycle    ;
\draw [color={rgb, 255:red, 208; green, 2; blue, 27 }  ,draw opacity=1 ]   (163.08,52.2) .. controls (135.5,68) and (110.5,70) .. (103.5,90) .. controls (96.5,110) and (96.5,143) .. (143,152.25) ;
\draw [shift={(138.27,64.02)}, rotate = 157.17] [fill={rgb, 255:red, 208; green, 2; blue, 27 }  ,fill opacity=1 ][line width=0.08]  [draw opacity=0] (12,-3) -- (0,0) -- (12,3) -- cycle    ;
\draw [shift={(102.63,123.72)}, rotate = 65.32] [fill={rgb, 255:red, 208; green, 2; blue, 27 }  ,fill opacity=1 ][line width=0.08]  [draw opacity=0] (12,-3) -- (0,0) -- (12,3) -- cycle    ;
\draw [color={rgb, 255:red, 208; green, 2; blue, 27 }  ,draw opacity=1 ]   (44.29,90.8) .. controls (70.5,92) and (60.73,113.93) .. (71.5,128) .. controls (82.27,142.07) and (102.5,159) .. (143,152.25) ;
\draw [shift={(59.39,96.38)}, rotate = 57.7] [fill={rgb, 255:red, 208; green, 2; blue, 27 }  ,fill opacity=1 ][line width=0.08]  [draw opacity=0] (12,-3) -- (0,0) -- (12,3) -- cycle    ;
\draw [shift={(95.42,147.62)}, rotate = 22.19] [fill={rgb, 255:red, 208; green, 2; blue, 27 }  ,fill opacity=1 ][line width=0.08]  [draw opacity=0] (12,-3) -- (0,0) -- (12,3) -- cycle    ;
\draw [color={rgb, 255:red, 208; green, 2; blue, 27 }  ,draw opacity=1 ]   (44.29,215.7) .. controls (48.56,189.22) and (51.79,173.56) .. (74.5,163) .. controls (97.21,152.44) and (95.5,181) .. (143,152.25) ;
\draw [shift={(48.92,192.81)}, rotate = 288.56] [fill={rgb, 255:red, 208; green, 2; blue, 27 }  ,fill opacity=1 ][line width=0.08]  [draw opacity=0] (12,-3) -- (0,0) -- (12,3) -- cycle    ;
\draw [shift={(100.73,164.76)}, rotate = 7.08] [fill={rgb, 255:red, 208; green, 2; blue, 27 }  ,fill opacity=1 ][line width=0.08]  [draw opacity=0] (12,-3) -- (0,0) -- (12,3) -- cycle    ;
\draw [color={rgb, 255:red, 208; green, 2; blue, 27 }  ,draw opacity=1 ]   (143,152.25) .. controls (72.5,198) and (142.5,225) .. (163.08,254.3) ;
\draw [shift={(120.66,211.85)}, rotate = 235.13] [fill={rgb, 255:red, 208; green, 2; blue, 27 }  ,fill opacity=1 ][line width=0.08]  [draw opacity=0] (12,-3) -- (0,0) -- (12,3) -- cycle    ;
\draw  [color={rgb, 255:red, 255; green, 255; blue, 255 }  ,draw opacity=1 ][fill={rgb, 255:red, 255; green, 255; blue, 255 }  ,fill opacity=1 ] (68,215) -- (110.5,215) -- (110.5,257) -- (68,257) -- cycle ;
\draw   (236.5,154.25) -- (309.93,53.21) -- (428.72,91.83) -- (428.7,216.74) -- (309.9,255.31) -- cycle ;
\draw [color={rgb, 255:red, 208; green, 2; blue, 27 }  ,draw opacity=1 ]   (329.79,154.95) .. controls (400.28,236.21) and (236.5,164.25) .. (236.5,154.25) ;
\draw [shift={(316.18,190.55)}, rotate = 194.19] [fill={rgb, 255:red, 208; green, 2; blue, 27 }  ,fill opacity=1 ][line width=0.08]  [draw opacity=0] (12,-3) -- (0,0) -- (12,3) -- cycle    ;
\draw [color={rgb, 255:red, 208; green, 2; blue, 27 }  ,draw opacity=1 ]   (309.69,255) .. controls (337.28,239.2) and (362.28,237.21) .. (369.28,217.21) .. controls (376.28,197.21) and (376.29,164.21) .. (329.79,154.95) ;
\draw [shift={(349.08,236.78)}, rotate = 155.13] [fill={rgb, 255:red, 208; green, 2; blue, 27 }  ,fill opacity=1 ][line width=0.08]  [draw opacity=0] (12,-3) -- (0,0) -- (12,3) -- cycle    ;
\draw [shift={(361.75,170.24)}, rotate = 49.93] [fill={rgb, 255:red, 208; green, 2; blue, 27 }  ,fill opacity=1 ][line width=0.08]  [draw opacity=0] (12,-3) -- (0,0) -- (12,3) -- cycle    ;
\draw [color={rgb, 255:red, 208; green, 2; blue, 27 }  ,draw opacity=1 ]   (428.49,216.42) .. controls (402.28,215.21) and (412.06,193.28) .. (401.29,179.21) .. controls (390.52,165.14) and (370.29,148.21) .. (329.79,154.95) ;
\draw [shift={(407.57,196.13)}, rotate = 77.08] [fill={rgb, 255:red, 208; green, 2; blue, 27 }  ,fill opacity=1 ][line width=0.08]  [draw opacity=0] (12,-3) -- (0,0) -- (12,3) -- cycle    ;
\draw [shift={(362.37,154.8)}, rotate = 13.37] [fill={rgb, 255:red, 208; green, 2; blue, 27 }  ,fill opacity=1 ][line width=0.08]  [draw opacity=0] (12,-3) -- (0,0) -- (12,3) -- cycle    ;
\draw [color={rgb, 255:red, 208; green, 2; blue, 27 }  ,draw opacity=1 ]   (428.51,91.51) .. controls (424.24,118) and (421,133.65) .. (398.29,144.21) .. controls (375.58,154.77) and (377.3,126.21) .. (329.79,154.95) ;
\draw [shift={(417.51,128.73)}, rotate = 299.64] [fill={rgb, 255:red, 208; green, 2; blue, 27 }  ,fill opacity=1 ][line width=0.08]  [draw opacity=0] (12,-3) -- (0,0) -- (12,3) -- cycle    ;
\draw [shift={(356.08,142.93)}, rotate = 349.79] [fill={rgb, 255:red, 208; green, 2; blue, 27 }  ,fill opacity=1 ][line width=0.08]  [draw opacity=0] (12,-3) -- (0,0) -- (12,3) -- cycle    ;
\draw [color={rgb, 255:red, 208; green, 2; blue, 27 }  ,draw opacity=1 ]   (329.79,154.95) .. controls (400.3,109.21) and (330.3,82.2) .. (309.73,52.9) ;
\draw [shift={(359.59,109.25)}, rotate = 248.94] [fill={rgb, 255:red, 208; green, 2; blue, 27 }  ,fill opacity=1 ][line width=0.08]  [draw opacity=0] (12,-3) -- (0,0) -- (12,3) -- cycle    ;
\draw  [color={rgb, 255:red, 255; green, 255; blue, 255 }  ,draw opacity=1 ][fill={rgb, 255:red, 255; green, 255; blue, 255 }  ,fill opacity=1 ] (404.8,92.21) -- (362.3,92.21) -- (362.31,50.21) -- (404.81,50.21) -- cycle ;
\draw [color={rgb, 255:red, 74; green, 144; blue, 226 }  ,draw opacity=1 ][line width=3]    (143,152.25) -- (232.48,153.54) -- (329.79,154.95) ;
\draw [shift={(175.84,152.72)}, rotate = 0.83] [fill={rgb, 255:red, 74; green, 144; blue, 226 }  ,fill opacity=1 ][line width=0.08]  [draw opacity=0] (18.75,-9.01) -- (0,0) -- (18.75,9.01) -- (12.45,0) -- cycle    ;
\draw [shift={(269.24,154.07)}, rotate = 0.83] [fill={rgb, 255:red, 74; green, 144; blue, 226 }  ,fill opacity=1 ][line width=0.08]  [draw opacity=0] (18.75,-9.01) -- (0,0) -- (18.75,9.01) -- (12.45,0) -- cycle    ;

\draw (233,127) node [anchor=north west][inner sep=0.75pt]   [align=left] {1};
\draw (160,35) node [anchor=north west][inner sep=0.75pt]   [align=left] {2};
\draw (30,77) node [anchor=north west][inner sep=0.75pt]   [align=left] {3};
\draw (31.29,218.7) node [anchor=north west][inner sep=0.75pt]   [align=left] {4};
\draw (159.08,256.3) node [anchor=north west][inner sep=0.75pt]   [align=left] {n};
\draw (303,258) node [anchor=north west][inner sep=0.75pt]   [align=left] {2};
\draw (430.49,219.42) node [anchor=north west][inner sep=0.75pt]   [align=left] {3};
\draw (432,81) node [anchor=north west][inner sep=0.75pt]   [align=left] {4};
\draw (304,35) node [anchor=north west][inner sep=0.75pt]   [align=left] {n};
\draw (87,19) node [anchor=north west][inner sep=0.75pt]   [align=left] {$\displaystyle F_{+}$};
\draw (372,21) node [anchor=north west][inner sep=0.75pt]   [align=left] {$\displaystyle F_{-}$};
\draw (154,183) node [anchor=north west][inner sep=0.75pt]   [align=left] {$\displaystyle p_{+}$};
\draw (300,107) node [anchor=north west][inner sep=0.75pt]   [align=left] {$\displaystyle p_{-}$};
\end{tikzpicture}
\end{center}
\caption{\label{capping} The capping path.}
\end{figure}

We recover the formula of Schrader, Shen, and Zaslow (Def. \ref{def: R}) by specializing 
to the linking skein, using $W$ as a background class, and taking the 
twisted spin structure of Lemma \ref{best twisted spin structure}.  
It follows that their operator annihilates the corresponding specialization of
$\Psi_{X, L}$ to the linking skein.

\begin{remark} 
It is easy to see and noted in \cite{Schrader-Shen-Zaslow} that 
$R_{\Gamma, F, v}$ depends of the choice of $v \in F$ only by scalar multiple. After the appearance of this article, the corresponding statement for $A_{\Gamma,F,v}$ was shown in \cite{HSZ}.
\end{remark}

\begin{corollary}
If $(X, L)$ is an exact filling, then $A_{\Gamma, F, v} \cdot 1 = 0 \in Sk(L)$. 
\end{corollary} 
\begin{proof}
An exact Lagrangian bounds no curves, so $\Psi_{X, L} = 1$. 
\end{proof}

\begin{remark}
Another obstruction to exact fillings, 
used in \cite{Treumann-Zaslow}, is that an exact filling with topology of a handlebody
contributes a $\mathbb{G}_m^g$ to the augmentation variety, so in particular the point count
of the augmentation variety over $\mathbb{F}_q$ is bounded below by $(q-1)^g$.  Here we are instead
quantizing the $\mathbb{G}_m^g$.  It would be interesting
to understand better the relationship between these approaches.  
\end{remark} 

\bibliographystyle{hplain}
\bibliography{skeinrefs}

\end{document}